\documentclass[letterpaper,12pt]{amsart}

\usepackage{amssymb}
\usepackage{amsmath,amscd}
\usepackage{rotating}

\numberwithin{equation}{section}

\newtheorem{lemma}{Lemma}[section]
\newtheorem{theorem}[lemma]{Theorem}
\newtheorem{proposition}[lemma]{Proposition}
\newtheorem{definition}[lemma]{Definition}
\newtheorem{corollary}[lemma]{Corollary}
\newtheorem{hypothesis}[lemma]{Hypothesis}

\theoremstyle{definition}

\newtheorem{remark}[lemma]{Remark}

\newcommand{\Gal}{\mathrm{Gal}}
\newcommand{\Hol}{\mathrm{Hol}}
\newcommand{\Hom}{\mathrm{Hom}}
\newcommand{\Aut}{\mathrm{Aut}}
\newcommand{\Perm}{\mathrm{Perm}}

\newcommand{\ord}{\mathrm{ord}}
\newcommand{\lcm}{\mathrm{lcm}}
\newcommand{\id}{\mathrm{id}}

\newcommand{\Z}{\mathbb{Z}}

\newcommand{\K}{\mathcal{K}}
\newcommand{\F}{\mathcal{F}}
\newcommand{\NN}{\mathcal{N}}

\newcommand{\onto}{\twoheadrightarrow}
\newcommand{\LRA}{\Leftrightarrow}

\begin{document}
\title[Hopf-Galois Structures of Squarefree Degree]
{Hopf-Galois Structures of Squarefree Degree}

\author{Ali A.~Alabdali}
 \thanks{The first-named author acknowledges support from The
    Higher Committee for Education Development in Iraq (PhD
    studentship D-11-1775).}


\address{(A.~Alabdali) Department of Mathematics, College of Education
  for Pure Science, University of Mosul, Mosul, Iraq.}
\email{aaab201@exeter.ac.uk}

\author{Nigel P.~Byott}
\address{(N.~Byott) Department of Mathematics, College of Engineering,
  Mathematics and Physical Sciences, University of Exeter, Exeter 
EX4 4QF U.K.}  
\email{N.P.Byott@exeter.ac.uk}

\date{\today}

\bibliographystyle{amsalpha}

\begin{abstract} 
Let $n$ be a squarefree natural number, and let $G$, $\Gamma$ be two
groups of order $n$. We determine the number of Hopf-Galois structures
of type $G$ admitted by a Galois extension of fields with Galois group
isomorphic to $\Gamma$. We give some examples, including a full
treatment of the case where $n$ is the product of three primes. 
\end{abstract}

\maketitle

\section{Introduction}

The notion of Hopf-Galois extension was introduced by Chase and
Sweedler \cite{CS}. (We recall the definition in \S\ref{HGS} below.)
This was in part motivated by the study of inseparable extensions of
fields, but is also of considerable interest for separable
extensions. Indeed, Greither and Pareigis \cite{GP} showed that a
given separable extension of fields $L/K$ may admit many Hopf-Galois
structures, and that finding them can be reduced to a group-theoretic
problem. If $H$ is a $K$-Hopf algebra giving a Hopf-Galois structure
on $L/K$, then, extending scalars to the algebraic closure $K^c$ of
$K$, we have an isomorphism of $K^c$-Hopf algebras $K^c \otimes_K H
\cong K^c[G]$, where $K^c[G]$ is the group ring of a group $G$ whose
order coincides with the degree of $L/K$. The isomorphism type of $G$
is called the {\em type} of the Hopf-Galois structure. Given abstract
groups $\Gamma$, $G$ of the same finite order $n$, we would like to
determine the number $e(\Gamma,G)$ of Hopf-Galois structures of type
$G$ on a Galois extension $L/K$ with Galois group $\Gal(L/K) \cong \Gamma$.

Since the work of Greither and Pareigis, there have been a number of
papers enumerating Hopf-Galois structures on Galois extensions of
various kinds: see for example \cite{pq, CC, Ch03, CRV-symalt, kohl,
  kohl16, NZ}. In \cite{AB}, we considered cyclic extensions of
squarefree degree $n$, obtaining a simple formula for the number of
Hopf-Galois structures whose type is any given group of order $n$. The
purpose of this paper is to extend the work in \cite{AB} by allowing
an arbitrary Galois group. Thus we determine
$e(\Gamma,G)$ for any two groups $\Gamma$, $G$ of squarefree order
$n$. The answer, which we state in Theorem \ref{thm-HGS} below once we
have developed the necessary notation, depends on an interplay of the
structures of the two groups, and turns out to be considerably more
complicated than when $\Gamma$ is cyclic.

Most of this paper will be taken up with the proof of Theorem
\ref{thm-HGS}, but in \S\ref{example} we present a fairly elaborate
example to indicate the complexities that can occur. Then, in \S\ref{special}, we consider
various special cases of Theorem \ref{thm-HGS}, some
of which recover results already in the literature. In particular, we
give in Theorem \ref{3-prime-thm} a full treatment of the case where
$n$ is the product of three primes. Some partial results for this case
were given in \cite{Ch03, kohl13, kohl16}.

The problem of enumerating Hopf-Galois structures on Galois extensions
of fields is closely related (but not equivalent) to the problem of
enumerating certain algebraic structures called skew braces, which
were introduced by Guarnieri and Vendramin \cite{GV} to study
set-theoretical solutions of the quantum Yang-Baxter equation. Skew braces
generalise the braces defined by Rump in \cite{Rump}. The connection
between braces and Hopf-Galois structures was first observed by
Bachiller \cite{Bachiller}, and the relationship between the two
enumeration problems was clarified in the appendix to \cite{SV}. In a
separate paper \cite{AB-braces}, we apply the results obtained here to
enumerate skew braces of squarefree order. 

Much of the content of this paper, along with that of \cite{AB},
appears in the PhD thesis \cite{thesis} of the first-named author. We
take the opportunity here to make some improvements in the exposition
and to correct some minor errors in \cite{thesis}.  

\section{Groups of squarefree order and main result} \label{statement}

A finite group in which all Sylow subgroups are cyclic is
necessarily metabelian, so such groups may be classified, see
e.g.~\cite{MM}. This applies in particular to groups of squarefree
order, and combining \cite[Lemmas 3.5, 3.6]{MM} in this case, we
obtain the following result (cf.~\cite[Lemma 3.2]{AB}).

\begin{lemma} \label{sf-class}
Let $n \geq 1$ be a squarefree integer. Then any group of order $n$
has the form 
$$   G(d,e,k)= \langle \sigma, \tau \colon \sigma^e=\tau^d=1_G, \tau
\sigma \tau^{-1} = \sigma^k \rangle $$ 
where $n=de$, $\gcd(d,e)=1$ and $\ord_e(k)=d$. Conversely, any choice
of $d$, $e$ and $k$ satisfying these conditions gives a group
$G(d,e,k)$ of order $n$. Moreover, two such groups $G(d,e,k)$ and
$G(d',e',k')$ are isomorphic if and only if $d=d'$, $e=e'$, and
$k$, $k'$  generate the same cyclic subgroup of $\Z_e^\times$.
\end{lemma}

Here, for any natural numbers $a$, $m$ with $\gcd(a.m)=1$, we write 
 $\ord_m(a)$ for the order of $a$ in the group $\Z_m^\times$ of units
in the ring $\Z_m$ of integers modulo $m$. 

We now fix a squarefree number $n$ and a group $G=G(d,e,k)$ of order
$n$, as in Lemma \ref{sf-class}. Let 
\begin{equation} \label{def-zg}
   z=\gcd(k-1, e),   \qquad g=e/z,  
\end{equation}
and for each prime $q \mid e$, let
 $  r_q = \ord_q(k)$. 
In general, $d$, $g$ and $z$ do not determine the $r_q$. Moreover,
$d$, $g$, $z$ and the $r_q$ do not determine the isomorphism type of
$G$. This is illustrated by the examples in \S\ref{example}.

Next let $\Gamma$ be another group of order $n$.  We will, as far as
possible, use corresponding Greek and Roman letters for quantities 
associated with $\Gamma$ and $G$, whilst keeping the notation
consistent with \cite{AB}. Thus we write
\begin{equation} \label{def-Gamma}
   \Gamma =    G(\delta,\epsilon,\kappa)= \langle s, t \colon
  s^\epsilon=t^\delta =1, t s t^{-1} = s^\kappa \rangle,  
\end{equation}
with $\delta \epsilon =n$ and
$\ord_\epsilon \kappa = \delta$, and we set 
\begin{equation} \label{def-zeta}
   \zeta=\gcd(\kappa-1, \epsilon),   \quad \gamma=\epsilon/\zeta,
   \quad \rho_q=\ord_q(\kappa) \mbox{ for primes } q \mid \epsilon.
\end{equation}

We consider the set 
$$ \K = \{ \kappa^r : r \in  \Z_\delta^\times\}. $$
By Lemma \ref{sf-class}, we may replace $\kappa$ in
(\ref{def-Gamma}) by any element of $\K$ without changing the
isomorphism type of $\Gamma$. The group $\Z_\delta^\times$ acts
regularly on $\K$ by exponentiation, so its subgroup
$$ \Delta:=\{m \in \Z_\delta^\times : m \equiv 1 \pmod{\gcd(\delta,d)}
\} $$ 
acts without fixed points on $\K$. The index of $\Delta$ in
$\Z_\delta^\times$ is
$$  w=\varphi(\gcd(\delta,d)),  $$
where $\varphi$ is the Euler totient function. Let $\kappa_1$, \ldots,
$\kappa_w$ be a system of orbit representatives of $\Delta$ on $\K$. 

We now define two sets of primes, depending on both $G$ and $\Gamma$:
$$ S =  \{\mbox{primes }q \mid \gcd(\gamma,g) : \rho_q=r_q >2 \}; $$ 
$$ T = \{\mbox{primes }q \mid \gcd(\gamma,g) : \rho_q=r_q=2 \}. $$

For $1 \leq h \leq w$, we also define 
 $$  S_h = \{ q \in S :  \kappa_h \equiv k \mbox{ or } \kappa_h \equiv k^{-1}
   \pmod{q} \}.  $$
We make some comments on the definitions of $\kappa_h$ and $S_h$ in Remarks
\ref{choose-kappa}, \ref{Sh-wd} \and \ref{choose-kappa-bis} below. 

For a natural number $m$, let $\omega(m)$ denote the number of
distinct prime factors of $m$. Thus, when $m$ is squarefree,
$2^{\omega(m)}$ is the number of positive integer factors of $m$.

We can now state our main result. 

\begin{theorem} \label{thm-HGS}
Let $G$ and $\Gamma$ be groups of squarefree order $n$, and
let $L/K$ be a Galois extension of fields with $\Gal(L/K) \cong
\Gamma$. With the above notation, 
the number $e(\Gamma,G)$ of Hopf-Galois structures of type $G$
on $L/K$ is given by
$$ e(\Gamma,G) = \begin{cases} 
 \displaystyle{\frac{2^{\omega(g)} \varphi(d) \gamma}{w} \left(
   \prod_{q \in T} \frac{1}{q} 
\right) \sum_{h=1}^w \prod_{q \in S_h}\frac{q+1}{2q}} & \mbox{ if }
 \gamma \mid e, \\ 
           0 & \mbox{ if } \gamma \nmid e. \end{cases} $$
\end{theorem}

\begin{remark} 
In \cite[Theorem 2]{AB}, we gave a formula, as a sum over
factorisations $n=dgz$ of $n$ into $3$ factors, for the total number
of Hopf-Galois structures on a cyclic extension of squarefree degree
$n$. Attempts to give a similar formula when the Galois group is an
arbitrary group $\Gamma$ of squarefree order $n$ lead to complicated
multiple sums which do not appear to admit any significant
simplification. We therefore do not give any results of this sort in
this paper. For some partial results in this direction, see
\cite[\S5.5]{thesis}
\end{remark}

\section{Background on Hopf-Galois Structures} \label{HGS}

Let $L/K$ be a finite extension of fields and let $H$ be a
cocommutative $K$-Hopf algebra with a $K$-linear action on $L$. We say
that $L/K$ is an $H$-Galois extension if the following
conditions are satisfied: 
\begin{itemize}
\item[(i)] $h \cdot (xy) = \sum_h (h_{(1)} \cdot x) (h_{(2)} \cdot y)$
  for all $h \in H$ and $x$, $y \in L$, where we use Sweedler's
  notation $h \mapsto \sum_h h_{(1)} \otimes h_{(2)}$ for the
  comultiplication of $H$;
\item[(ii)] $h \cdot 1=\varepsilon(h) 1$, where $\varepsilon:H \to K$
  is the counit of $H$;
\item[(iii)] the $K$-linear map $\theta:L \otimes_K L \to \Hom_K(L,L)$
 is bijective, where $\theta(x \otimes y)(h)= x (h \cdot y)$. 
\end{itemize}
A Hopf-Galois structure on $L/K$ consists of a cocommutative $K$-Hopf
algebra $H$, together with an action of $H$ on $L$ making $L/K$ into an
$H$-Galois extension. 

Greither and Pareigis \cite{GP} showed how all Hopf-Galois structures
on a finite separable field extension $L/K$ can be described in terms
of group theory. We consider here only the case when $L/K$ is a Galois
extension in the classical sense (i.e., normal as well as
separable). Let $\Gamma=\Gal(L/K)$. Then the Hopf-Galois structures on
$L/K$ correspond bijectively to the regular subgroups $G \subset
\Perm(\Gamma)$ which are normalised by the group $\lambda(\Gamma)$ of
left translations by $\Gamma$. Here $\Perm(\Gamma)$ is the group of
permutations of the underlying set of $\Gamma$. The Hopf algebra
acting on $L$ in the Hopf-Galois structure corresponding to $G$ is 
$H=L[G]^\Gamma$, the algebra of $\Gamma$-fixed points of $L[G]$, where $\Gamma$
acts simultaneously on $L$ as field automorphisms and on $G$ as
conjugation by left translations inside $\Perm(\Gamma)$. The {\em
  type} of a Hopf-Galois structure is by definition the isomorphism
type of the corresponding group $G$.

Given abstract groups $\Gamma$ and $G$ of the same (finite) order, we
denote by $e(\Gamma,G)$ the number of Hopf-Galois structures of type
$G$ on a Galois extension $L/K$ with $\Gal(L/K) \cong \Gamma$. By the
theorem of Greither and Pareigis, this is just the number of regular
subgroups of $\Perm(\Gamma)$ which are isomorphic to $G$ and are
normalised by $\lambda(\Gamma)$. Equivalently, $e(\Gamma,G)$ is the
number of $\Aut(G)$-orbits of regular embeddings $\alpha:G \to
\Perm(\Gamma)$ whose image $\alpha(G)$ is normalised by
$\lambda(\Gamma)$.  As shown in \cite{unique}, there is a bijection
between the regular embeddings $\alpha:G \to \Perm(\Gamma)$ and the
regular embeddings $\beta: \Gamma \to \Perm(G)$, under which
$\alpha(G)$ is normalised by $\lambda(\Gamma)$ if and only of
$\beta(\Gamma)$ is contained in the normaliser $\Hol(G)=\lambda(G)
\rtimes \Aut(G)$ of $G$ in $\Perm(G)$. Thus $e(\Gamma,G)$ is the
number of $\Aut(G)$-orbits of regular embeddings $\Gamma \to \Hol(G)$,
or, equivalently,
\begin{equation} \label{HGS-count-formula}
 e(\Gamma,G) = 
   \frac{|\Aut(\Gamma)|}{|\Aut(G)|} e'(\Gamma,G), 
\end{equation}
where $e'(\Gamma,G)$ is the number of regular subgroups of $\Hol(G)$
isomorphic to $\Gamma$. As $\Hol(G)$ is usually much smaller than
$\Perm(\Gamma)$, it is in practice easier to calculate $e(\Gamma,G)$
by working in $\Hol(G)$ and using (\ref{HGS-count-formula}), rather
than by working directly in $\Perm(\Gamma)$. 

\section{The holomorph of $G$}

Let $G=G(d,e,k)$ be a group of squarefree order $n=de$ as in 
Lemma \ref{sf-class}. In this section, we recall from \cite{AB} some
results on $G$, $\Aut(G)$ and $\Hol(G)$. 

By \cite[Prop.~3.5]{AB}, the centre $Z(G)$ and commutator subgroup
$G'$ of $G$ are respectively the cyclic groups $\langle \sigma^g
\rangle$ of order $z$ and $\langle \sigma^z \rangle$ of order $g$,
where $z$, $g$ are defined in (\ref{def-zg}).

We write elements of $\Hol(G)=G \rtimes \Aut(G)$ as $[x,\alpha]$,
where $x \in G$ and $\alpha \in \Aut(G)$. The multiplication in
$\Hol(G)$ is given by
\begin{equation} \label{hol-mul}
 [x,\alpha] [x',\alpha'] = [x \alpha(x') , \alpha \alpha'], 
\end{equation}
and the action of $\Hol(G)$ on $G$ is given by
$$  [x,\alpha] \cdot y = x \alpha(y). $$
Thus the normal subgroup $G$ in $\Hol(G)=G \rtimes \Aut(G)$ identifies as the
group of left translations in $\Perm(G)$. 

In \cite[Lemma 4.1]{AB} we determined $\Aut(G)$:

\begin{lemma}  \label{Aut-G}
$\Aut(G) \cong \Z_g \rtimes \Z_e^\times$, and $\Aut(G)$ is generated by
  $\theta$ and by $\phi_s$ for $s \in \Z_e^\times$ where
$$ \theta(\sigma) = \sigma, \qquad \theta(\tau) = \sigma^z \tau $$
and 
$$ \phi_s(\sigma) = \sigma^s, \qquad \phi_s(\tau)= \tau. $$
These satisfy the relations  
$$ \theta^g = \id_G, \qquad \phi_s \phi_t= \phi_{st}, \qquad \phi_s
\theta \phi_s^{-1}=\theta^s. $$ 
\end{lemma}

An arbitrary element of $\Hol(G)$ therefore has the form $[\sigma^u
  \tau^f, \theta^v \phi_t]$ for $u \in \Z_g$, $f \in \Z_d$, $v \in
\Z_g$ and $t \in \Z_e^\times$. We record some consequences of this.

\begin{proposition} \label{proj-hom} \ 

\begin{itemize}
\item[(i)] The map $\Hol(G) \to \{ \phi_s : s \in \Z_e^\times\} \cong
  \Z_e^\times$, given by $[\sigma^u \tau^f, \theta_v \phi_t] \mapsto
  \phi_t$, is a group homomorphism.
\item[(ii)] The map $\Hol(G) \to \langle \tau \rangle \cong C_d$,
  given by $[\sigma^u \tau^f, \theta_v \phi_t] \mapsto \tau^f$, is a
  group homomorphism
\end{itemize}
\end{proposition}
\begin{proof}
\noindent (i) The map is the composite of the two quotient homomorphisms
$\Hol(G) \onto \Aut(G) \onto \Z_e^\times$.

\noindent (ii) This follows from the fact that every automorphism of $G$ induces
  the trivial automorphism on the quotient $G/\langle \sigma \rangle
  \cong C_d$ of $G$, as explained in the proof of
  \cite[Prop.~4.3]{AB}. (Note that the map $[\sigma^u \tau^f, \theta^v
    \phi_t] \mapsto \sigma^u \tau^f$ is {\em not} a homomorphism.)
\end{proof}

In the next section, we will consider pairs of elements of $\Hol(G)$
of the following special form:
\begin{equation} \label{def-XY}
  X=[ \sigma^a, \theta^c], \qquad Y=[\sigma^u \tau,
    \theta^v \phi_{t}] .
\end{equation}
Thus $X$ does not involve $\tau$ or any $\phi_s$, and $\tau$ occurs
in $Y$ with exponent $1$. 

As $\theta(\sigma)=\sigma$, it is clear that  
\begin{equation} \label{pow-X}
  X^j= [\sigma^{aj}, \theta^{cj}], \qquad X^j \cdot \sigma^i =
  \sigma^{aj+i}  \mbox{ for $i$, $j \geq 0$.}
\end{equation}

\begin{lemma}  \label{Y-pow}
For $Y$ as in (\ref{def-XY}) and $j\geq 0$, we have 
$$   Y^j = [ \sigma^{A(j)} \tau^j, \theta^{vS(t,j)} \phi_{t^j}] $$
where 
\begin{equation} \label{def-A}
  A(j) = uS(tk,j) + vzkT(k,t,j),  
\end{equation}
with
$$     S(m,j)= \sum_{i=0}^{j-1} m^i. $$
and
$$  T(k,t,j) = \sum_{h=0}^{j-1} S(t,h) k^{h-1} \mbox{ for } j \geq 1,
  \qquad  T(k,t,0)=0. $$
\end{lemma}
\begin{proof} See \cite[Lemma 4.2]{AB}.
\end{proof}

Using the fact that $k^d \equiv 1 \pmod{e}$, we can determine the
residue classes of the sums in Lemma \ref{Y-pow} at primes dividing $e$
in the case that $j$ is a multiple of $d$.

\begin{proposition} \label{ST-cong}
Let $q$ be a prime dividing $e$. Then we have the following
congruences mod $q$. (We omit the modulus for brevity.)
\begin{itemize}
\item[(i)] For any $s$, $i \in \Z$ with $i \geq 0$, we have 
$$   S(s,di) \equiv \begin{cases} 
   di & \mbox{if } s \equiv 1; \cr \cr
   \displaystyle{\frac{s^{di}-1}{s-1}} &
   \mbox{otherwise}. \end{cases} $$
\item[(ii)] If $k \not \equiv 1$
  then, for any $t$, $i \in \Z$ with $i \geq 0$, we have
$$ T(k,t,di) \equiv \begin{cases}
    \displaystyle{\frac{di}{k(k-1)}} & \mbox{if } t \equiv 1; \cr \cr  
    \displaystyle{\frac{di}{k(t-1)}} & \mbox{if } tk \equiv 1; \cr \cr
   \displaystyle{\frac{ (t^{di}-1)}{k(t-1)(tk-1)}} &
   \mbox{otherwise}. \end{cases}  $$ 
\end{itemize}
\end{proposition}
\begin{proof} This is \cite[Prop.~5.2]{AB}.
\end{proof}

\section{Regular subgroups in $\Hol(G)$}

As before, let $G=G(d,e,k)$ be a group of squarefree order $n$ as in
Lemma \ref{sf-class}, and let $\Gamma=G(\delta,\epsilon, \kappa)$ be a
second group of order $n$, as in (\ref{def-Gamma}).  In this section,
we begin our investigation of regular subgroups of $\Hol(G)$
isomorphic to $\Gamma$.

It will be convenient to modify the presentation (\ref{def-Gamma}) of
$\Gamma$.  Set $X=s^\zeta$ and $Y=ts^\gamma$. As $s^\gamma$ generates
the centre of $\Gamma$, we have
\begin{equation} \label{new-pres}
   \Gamma = \langle X, Y \colon X^\gamma=Y^{\zeta\delta} =1, YXY^{-1}
   = X^\kappa \rangle.
\end{equation}
\begin{remark}
In (\ref{new-pres}), $\kappa$ is viewed as an element of
$\Z_\gamma^\times$, with $\gcd(\kappa-1, \gamma) =1$. This is
equivalent to viewing $\kappa$ as an element of $\Z_\epsilon^\times
\cong \Z_\gamma^\times \times \Z_\zeta^\times$ with $\gcd(\kappa-1,
\epsilon) =\zeta$, as in (\ref{def-Gamma}) and (\ref{def-zeta}).
\end{remark}

\begin{proposition} \label{d-div-gam-del}
If $\Hol(G)$ contains a regular subgroup $\Gamma^* \cong \Gamma$ 
then $\gamma \mid e$. Moreover, if $X$ and
$Y$ are generators of $\Gamma^*$ satisfying the relations in
(\ref{new-pres}) then the subgroup $\langle X, Y^d \rangle$ of $
\Gamma^*$ of order $e$ acts regularly on the subset $\{ \sigma^m : m
\in \Z\}$ of $G$.
\end{proposition}
\begin{proof}
Let $\Gamma^*=\langle X, Y \rangle \cong \Gamma$ be a regular
subgroup, where the generators $X$, $Y$ satisfy the relations in
(\ref{new-pres}).  Then the commutator subgroup of $\Gamma^*$ is
generated by $X$. Thus $X$ lies in the commutator subgroup of
$\Hol(G)$, and therefore in the kernel of the homomorphism $\Hol(G)
\to \langle \tau \rangle$ of Proposition \ref{proj-hom}(ii). Hence $X$
cannot involve $\tau$. But $\Gamma^*$ is regular, so $X^i Y^f \cdot
1_G=\tau$ for some $i$, $f$. It follows that $Y^f=[\sigma^a \tau,
  \psi]$ for some $a \in \Z$ and some $\psi \in \Aut(G)$. Thus, for $j
\geq 0$, we have $Y^{fj} \cdot 1_G \in \{ \sigma^m : m \in \Z\}$ if
and only if $d \mid j$. As $Y$ has order $\zeta \delta$, it follows
that $ d \mid \zeta \delta $. Since $ de= n= \delta \gamma \zeta $,
this is equivalent to $ \gamma \mid e $. Moreover, the subgroup $
\langle X, Y^d \rangle $ has order $ e $ and acts without fixed points
on the subset $\{ \sigma^m : m \in \Z\}$ of $ G $, which has
cardinality $e$. Hence this action is regular.
\end{proof}

In view of the first assertion of Proposition \ref{d-div-gam-del},
we now impose the following hypothesis:

\begin{hypothesis} \label{HGS-exist}
The groups $G$ and $\Gamma$ are chosen so that $\gamma \mid e$. 
\end{hypothesis}

Note that, as shown above,
the condition $\gamma \mid e$ is equivalent to $d \mid \zeta \delta$.

Recall that in \S\ref{statement} we defined the set $\K$ and the group
$\Delta$, and chose a system $\kappa_1, \ldots, \kappa_w$ of orbit
representatives for the action of $\Delta$ on $\K$, with $w=
\varphi(\gcd(\delta,d))$.

\begin{remark} \label{choose-kappa}
If $\gcd(\delta,d)=1$ or $2$, then $w=1$ and we may take
$\kappa_1=\kappa$. If $\gcd(\delta,d)>2$ then 
$w$ is even and $-1 \not \in \Delta$. We may then choose the
$\kappa_h$ so that $\kappa_{w+1-h}=\kappa_h^{-1}$ for $1 \leq h \leq w$. 
\end{remark}

\begin{lemma} \label{h unique}
Let $ \Gamma^* $ be a regular subgroup of $ \Hol(G) $ isomorphic to $
\Gamma $. Then there is a unique $h$ with $1 \leq h \leq w$ such
that $ \Gamma^* $ is generated by a pair of elements $X$, $Y$
in the special form of (\ref{def-XY}) which satisfy the relations
\begin{equation} \label{XY-orders}
X^\gamma = Y^{\zeta\delta} = 1, 
\end{equation}
\begin{equation} \label{XYkap-h}
     YXY^{-1} = X^{\kappa_h}.
\end{equation}
Indeed, $\Gamma^*$ contains exactly $\gamma \varphi(e) w /
\varphi(\delta)$ such pairs of generators. 
\end{lemma}
\begin{proof}
Since $ \Gamma^* \cong \Gamma $, we can find some pair of generators $
X, Y $ satisfying the relations in (\ref{new-pres}). As in the proof
of Proposition \ref{d-div-gam-del}, $ \tau $ cannot occur in $ X $,
and there is some $ f \in \Z $ so that $ \tau $ occurs with exponent 1
in $Y^f$. Then $ \gcd(f, d) = 1 $. Since $d \mid \zeta \delta$ and
$\zeta \delta$ is squarefree, we may choose $f$ with $\gcd(f,
\zeta\delta) = 1$. Thus $Y^f$ has order $\zeta\delta$. We replace $Y$
by $Y^f$. Our new pair of generators $X$, $Y$ satisfy the relations in
(\ref{new-pres}), but with $\kappa$ replaced by $\kappa^f$. Moreover,
$Y$ is as in (\ref{def-XY}) and $X$ has the form 
$$ X= [ \sigma^a, \theta^c \phi_s].  $$
We claim that $s \equiv 1 \pmod{e}$, so that $X$ is also as in
(\ref{def-XY}).  Applying Proposition \ref{proj-hom}(i) to the
relations $X^\gamma = 1$, $YX = X^\kappa Y$, we obtain $\phi_s^\gamma=
1$, $ \phi_{t} \phi_s= \phi_s^\kappa \phi_t$. Thus $s^\gamma \equiv 1
\equiv s^{\kappa -1} \pmod{e} $.  As $\gcd(\gamma, \kappa -1) = 1$, it
follows that $s \equiv 1 \pmod{e}$ as claimed.

We next consider further changes to our generators $X$, $Y$ so that
they are still of the form (\ref{def-XY}), and satisfy the relations
in (\ref{new-pres}) except that $\kappa$ may be replaced by some other
element of $\mathcal{K}$. Thus we can replace $X$ by $x = X^i$ for any
$i \in \Z_\gamma^\times$, and $Y$ by $y = X^j Y^m$ for any $j \in
\Z_\gamma$ and $m \in \Z_{\zeta \delta}^\times$ with $ m \equiv 1
\pmod{d}$. (The last condition ensures that $\tau$ still occurs with
exponent 1 in $y$.) Then
$$ yxy^{-1} = (X^j Y^m) X^i (Y^{-m} X^{-j}) = X^{i\kappa^m} = x^{\kappa^m} $$
As $\kappa \in \Z_\gamma^\times$ has order $\delta$, we have 
\begin{eqnarray*}
 \{ \kappa^m \colon m \equiv 1 \pmod d \} 
      & = &  \{ \kappa^r \colon r \equiv 1 \pmod {\gcd(\delta,d)} \} \\ 
        & = &  \{ \kappa^r \colon r \in \Delta \}. 
\end{eqnarray*}
Thus, given $r \in \Z_\delta^\times$, the original generators $X$, $Y$
can be replaced by generators $x$, $y$ of the form (\ref{def-XY})
satisfying the relations $x^\gamma = y^{\zeta\delta} = 1$, $yx =
x^{\kappa^r}y$ if and only if $r \in \Delta$. It follows that there is
a unique $h$ such that $\Gamma^*$ contains a pair of generators
satisfying (\ref{def-XY}), (\ref{XY-orders}) and (\ref{XYkap-h}).

Finally, suppose $X$, $Y$ is one such pair of generators. We determine
when $x = X^i$, $y = X^j Y^m$ is another. Since $X$ has order
$\gamma$, (\ref{XYkap-h}) will hold if and only if $\kappa^m \equiv
\kappa \pmod \gamma$.  As $\kappa$ has order $\delta$, this is
equivalent to $m \equiv 1 \pmod \delta$.  Since we already have the
condition $m \equiv 1 \pmod{d}$, and $d \mid \zeta \delta$ by
Hypothesis \ref{HGS-exist}, we must choose $m \in \Z_{\zeta
  \delta}^\times$ with $ m \equiv 1 \pmod {\lcm (\delta,d)}$. Hence
the number of choices for $m$ is $\varphi(\zeta\delta)/\varphi(\lcm(
\delta,d))$. Now, since $d$ and $\delta$ are squarefree, we have
$$ \varphi(\lcm(\delta,d)) = \varphi \left(\frac{d \delta}{\gcd(\delta,
  d)} \right) = 
   \frac{\varphi(d)\varphi(\delta)}{\varphi(\gcd(\delta,d))} =
   \frac{\varphi(d)\varphi(\delta)}{w}. $$ 
So the number of choices for $m$ is 
$\varphi(\zeta \delta)w/(\varphi(d)\varphi(\delta))$. There are
$\varphi(\gamma)$ 
choices for $i \in \mathbb{Z}_\gamma^\times$ and $\gamma$ choices for
$j \in\mathbb{Z}_\gamma $. Therefore, the number of pairs $X$, $Y$ of
generators of $\Gamma^*$ satisfying 
(\ref{def-XY}), (\ref{XY-orders}) and (\ref{XYkap-h}) is  
$$ \frac{\varphi(\gamma) \gamma \varphi(\zeta \delta) w}{\varphi(d)
  \varphi(\delta)} 
 =  \frac{\gamma \varphi(n) w}{\varphi(d) \varphi(\delta)} 
  =  \frac{\gamma \varphi(e) w}{\varphi(\delta)}. $$
\end{proof}

Lemma \ref{h unique} shows that the regular subgroups of $\Hol(G)$
isomorphic to $\Gamma$ fall into $w$ disjoint families $\F_1, \ldots,
\F_w$, where $\F_h$ consists of those subgroups with a pair of
generators $X$, $Y$ satisfying (\ref{def-XY}) and (\ref{XY-orders}), and
satisfying (\ref{XYkap-h}) for the orbit representative $\kappa_h$. These
families therefore correspond to the orbits of $\Delta$ on
$\mathcal{K}$. (We will see in Lemma \ref{count-Nh} that all the $\F_h$ are nonempty.)
However, not every pair of elements $(X, Y)$ satisfying
(\ref{def-XY}), (\ref{XY-orders}) and (\ref{XYkap-h}) will generate a
regular subgroup. We now characterise those that do.

\begin{lemma} \label{XY}
Let $1 \leq h \leq w$ and let $X$, $Y \in \Hol(G)$ be as in
(\ref{def-XY}). Suppose further that $X$ and $Y$ satisfy
(\ref{XYkap-h}).  Then the group $\langle X, Y \rangle \subseteq \Hol(G)$ is
regular on $G$ if and only if
\begin{itemize}
\item[(i)] the group $ \langle X \rangle $ acts regularly on the
  subset $ \{ \sigma^{em/\gamma} \colon m \in \Z \} $ of $G$ of
  cardinality $ \gamma $; 
\item[(ii)] the group $ \langle X, Y^d \rangle $ acts transitively on
  the subset $ \{ \sigma^m \colon m \in \Z \} $ of $ G $ of
  cardinality $e$; 
\item[(iii)] $Y^{\zeta \delta} = 1$.
\end{itemize}
If these conditions hold, then (\ref{XY-orders}) also holds, and $\langle
X, Y \rangle \cong \Gamma$. 
\end{lemma}
\begin{proof}
Suppose $ \Gamma^*= \langle X, Y \rangle $ is regular. 
Then (i) follows from Proposition \ref{d-div-gam-del}, so in
particular $ X $ has order $ \gamma $. As $ |\Gamma^*| = n = \gamma
\zeta \delta $, (iii) follows. Since $ Y^d $ does not involve $
\tau $, the subgroup $ \langle X, Y^d \rangle $ of index $ d $ acts on
$ \{ \sigma^m \colon m \in \Z \} $. This action must be regular,
and hence transitive. Thus (ii) holds. 

Conversely, suppose $X$, $Y$ satisfy (i), (ii) and (iii), and let $
\Gamma^*= \langle X, Y \rangle $. By (i), $X$ has order $\gamma$. It
then follows from (iii) that (\ref{XY-orders}) is satisfied.  In
particular, $\Gamma^*$ has order dividing $n$. We claim that
$\Gamma^*$ is transitive on $G$. This will show that $\Gamma^*$ is
regular and isomorphic to $\Gamma$.  Let $\sigma^i \tau^j$ be an
arbitrary element of $G$. As $\tau$ occurs in $Y$ with exponent $1$,
we have $ Y^{-j} \cdot \sigma^i \tau^j = \sigma^m$ for some $m$.  It
follows from (ii) that $\sigma^i \tau^j$ is in the same
$\Gamma^*$-orbit as $1_G$. Hence $\Gamma^*$ is transitive as claimed.
\end{proof}

\begin{definition} \label{def-N}
For $1 \leq h \leq w$, let 
$$ \NN_h  \subset \Z_e^\times \times \Z_e \times \Z_g \times \Z_e \times \Z_g$$  
be the set of quintuples  $(t, a,c,u,v)$
such that the corresponding elements 
$X = [\sigma^a, \theta^c]$, $Y=[\sigma^u \tau, \theta^v \phi_t]$
satisfy (\ref{XYkap-h}) and generate a regular
subgroup isomorphic to $\Gamma$.
\end{definition}

Thus $(t,a,c,u, v) \in \NN_h$ if and only if $X$ and $Y$ satisfy
(\ref{XYkap-h}) and conditions (i), (ii), (iii) of Lemma \ref{XY}. 

By the last assertion of Lemma \ref{h unique},
the number $e'(\Gamma,G)$ of regular subgroups isomorphic to
$\Gamma$ in $\Hol(G)$ is given by 
\begin{equation} \label{sum-Nh}
  e'(\Gamma,G) = \sum_{h=1}^w |\F_h| = \frac{\varphi(\delta)}{\gamma
\varphi(e)w} \sum_{h=1}^w |\NN_h|. 
\end{equation}

\section{Quintuples giving regular subgroups}

Suppose that Hypothesis \ref{HGS-exist} holds.  Our goal in this
section is to give an explicit characterisation of the set $\NN_h$ in
Definition \ref{def-N} by means of congruences at each prime $q \mid
e$.  We will achieve this in Lemma \ref{quin}, after defining some
further notation and proving a number of preliminary propositions.

Recall that $r_q=\ord_q(k)$ for each prime $q \mid e$, and
$\rho_q=\ord_q(\kappa)$ for each $q \mid \epsilon$. Thus we have
$$ r_q = 1 \LRA q \mid z, \qquad r_q \mid \gcd(d,q-1), \qquad 
   \lcm\{r_q : q \mid g \} = d, $$
and similarly for the $\rho_q$. 
\begin{remark} \label{g-odd}
If $q \mid g$, we have $1<r_q \leq q-1$, so $q \geq 3$. Thus
$g$ is always odd.
\end{remark}

We divide the set of primes $q \mid e$
into six subsets (any of which may be empty).

\begin{definition} \label{subsets}
 $$ P = \{\mbox{primes }q \mid \gcd(\gamma,z)  \}; $$
$$  Q = \{\mbox{primes }q \mid \gcd(\zeta \delta,z)  \}; $$
$$  R= \{\mbox{primes }q \mid \gcd(\gamma,g) : \rho_q \neq r_q \}; $$
$$  S =  \{\mbox{primes }q \mid \gcd(\gamma,g) : \rho_q=r_q >2 \}; $$ 
$$  T = \{\mbox{primes }q \mid \gcd(\gamma,g) : \rho_q=r_q=2 \}; $$
$$  U = \{\mbox{primes }q \mid \gcd(\zeta \delta,g)  \}. $$
\end{definition}
For $1 \leq h \leq w$, we further define some subsets of $S$ depending on $h$.
\begin{definition} \label{SSh} 
$$   S_h^+ = \{ q \in S :  \kappa_h \equiv k \pmod{q} \}; $$
$$  S_h^- = \{ q \in S : \kappa_h \equiv k^{-1} \pmod{q} \}. $$
$$   S_h = S_h^+ \cup S_h^-;$$
$$ S'_h = S \backslash S_h. $$ 
\end{definition}
The sets $S$, $T$ and $S_h$ are as already defined in \S\ref{statement}.

\begin{remark} \label{Sh-wd}
Let $q \in S$. Then, as $\rho_q \mid \gcd(\delta,d)$, we have
$\kappa_h^m \equiv \kappa_h \pmod{q}$ for $m \equiv 1
\pmod{\gcd(\delta,d)}$. This shows that the sets $S_h^+$ and $S_h^-$
are independent of the choice of orbit representatives
$\kappa_h$. Also, $q \in T$ if and only if $k \equiv \kappa \equiv -1
\not \equiv 1 \pmod{q}$. In particular, $T=\emptyset$ unless $d$ and
$\delta$ are both even.
\end{remark}

\begin{remark} \label{choose-kappa-bis}
If $w>1$ and the $\kappa_h$ are chosen as in
Remark \ref{choose-kappa}, then $S_h^-=S^+_{w+1-h}$ for each $h$. 
\end{remark}

The set $S$ is the union of the sets $S_h^+$, but in general this
union is not disjoint. 

\begin{proposition} \label{num-h}
Let $q \in S$. Then there are exactly $w/\varphi(r_q)$
values of $h$ with $q \in 
S_h^+$. In particular, $q \in S^+_h$ for a unique $h$ if
and only if $r_q=\gcd(\delta,d)$ or $r_q = \frac{1}{2}
\gcd(\delta,d)$.
\end{proposition}
\begin{proof}
We may decompose the group
$\Z_\delta^\times$ as 
$$ \Z_\delta^\times = \Z_{\gcd(\delta,d)}^\times \times
\Z_{\gcd(\delta,e)}^\times \cong \Z_{\gcd(\delta,d)}^\times \times
\Delta, $$
so that the orbits of $\K$ under $\Delta$ correspond to the elements
of $\Z_{\gcd(\delta,d)}^\times$. Explicitly, if the orbit represented
by $\kappa_h$ corresponds to $j \in \Z_{\gcd(\delta,d)}^\times$, then
$\kappa_h \equiv \kappa^J \pmod{\epsilon}$ for some $J \equiv j
  \pmod{\gcd(\delta,d)}$ with $\gcd(J,\delta)=1$. Now let also
  $\kappa_{h'}\equiv \kappa^{J'}$ represent the orbit corresponding to
  $j'$. Then
\begin{eqnarray*}
  \kappa_h \equiv \kappa_{h'} \pmod{q} & \Leftrightarrow & \kappa^J
  \equiv \kappa^{J'} \pmod{q} \\
 & \Leftrightarrow & J \equiv J' \pmod{\rho_q} \\
 & \Leftrightarrow & j \equiv j' \pmod{\rho_q} 
\end{eqnarray*}
since $r_q=\rho_q$ divides $\gcd(\delta,d)$. As the canonical
surjection $\Z_{\gcd(\delta,d)}^\times \to \Z_{r_q}^\times$ has kernel
of cardinality $w/\varphi(r_q)$, it follows that $j \equiv j'
\pmod{r_q}$ for precisely $w/\varphi(r_q)$ of the $w$ possible values
of $j'$. Thus $q \in S_h^+$ for precisely $w/\varphi(r_q)$ values of
$h$. 

In particular, as $w= \varphi(\gcd(\delta,d))$, there is a unique $h$ with 
$q \in S_h^+$ if and only if $\varphi(r_q)=\varphi(\gcd(\delta,d))$. As $r_q
\mid \gcd(\delta,d)$, and $\gcd(\delta,d)$ is squarefree, this occurs
if and only if $\gcd(\delta,d)=r_q$ or $2r_q$.
\end{proof}

We now fix $h \in \{1, \ldots, w\}$. To simplify the notation, we
write $\kappa$ in place of $\kappa_h$ for the rest of this section.

Let 
$$ (t,a,c,u,v) \in \Z_e^\times \times \Z_e \times \Z_g \times \Z_e
\times \Z_g. $$
We wish to determine when this quintuple belongs to $\NN_h$. Equivalently, we
wish to determine when the elements $X=[ \sigma^a, \theta^c]$ and
$Y=[\sigma^u \tau,  \theta^v \phi_{t}]$ of $\Hol(G)$ generate a
subgroup $\Gamma^*=\langle X, Y \rangle$ in $\F_h$.
  
To do so, we translate the conditions in Lemma \ref{XY} into
congruence conditions on $t$, $a$, $c$, $u$, $v$. In the following,
all congruences are modulo the relevant prime $q$ unless otherwise
indicated.

\begin{proposition} \label{X-ord}
Condition (i) of Lemma \ref{XY} is equivalent to the following
conditions at each prime $q \mid e$:
\begin{equation} \label{X-reg-a}
  \mbox{ if } q \mid \gamma \mbox{ then } a \not \equiv 0; 
\end{equation}
\begin{equation} \label{X-ord-a}
  \mbox{ if } q \mid \gcd(\zeta \delta,e) \mbox{ then } a \equiv 0;
\end{equation}
\begin{equation} \label{X-ord-c}
  \mbox{ if } q \mid  \gcd(\zeta \delta,g) \mbox{ then }
  c \equiv 0.  
\end{equation}
\end{proposition} 
\begin{proof} Condition (i) of Lemma \ref{XY} is equivalent to 
$$ X^j \cdot 1_G= 1_G \Leftrightarrow \gamma \mid j, \qquad X^\gamma= 1  $$ 
and so, using (\ref{pow-X}), to 
$$ e \mid aj \Leftrightarrow \gamma \mid j, \qquad c \gamma \equiv 0
\pmod{g}. $$ 
(Recall that $\gamma \mid e$ by Hypothesis \ref{HGS-exist}). These in
turn are equivalent to (\ref{X-reg-a}), (\ref{X-ord-a}) and (\ref{X-ord-c}).
\end{proof}

\begin{proposition} \label{XY-com}
Assume that condition (i) of Lemma \ref{XY} holds. 
Then (\ref{XYkap-h}) is equivalent to the following
conditions for each prime $q$: 
\begin{equation} \label{com-sz}
    \mbox{if } q \mid \gcd(\gamma,z)  \mbox{ then } t
    \equiv \kappa; 
\end{equation}
\begin{eqnarray} \label{com-sg}
    \lefteqn{\mbox{if } q \mid \gcd(\gamma,g) \mbox{
        then either}  } \\ 
       & \mbox{(i)} &  t \equiv \kappa, \quad c \equiv \lambda a,
    \mbox{ or} 
     \nonumber \\  
       & \mbox{(ii)} &  tk \equiv \kappa, \quad c \equiv 0, \nonumber
\end{eqnarray}
where $\lambda = z^{-1}(k-1) \in \Z_g^\times$.
\end{proposition} 
\begin{proof}
Our assumption means that (\ref{X-reg-a}), (\ref{X-ord-a}),
(\ref{X-ord-c}) hold.

Writing (\ref{XYkap-h}) as $YX=X^\kappa Y$ and expanding using
(\ref{hol-mul}) and (\ref{pow-X}), we obtain
$$ [\sigma^{u+atk} \tau, \theta^{v+ct} \phi_{t}]= 
       [\sigma^{a \kappa+u+zc\kappa} \tau, \theta^{c \kappa+v} \phi_{t}]. $$
This is equivalent to the two congruences 
\begin{equation} \label{cong-a}
    a(tk-\kappa) \equiv zc\kappa \pmod{e},
\end{equation} 
\begin{equation} \label{cong-b}
     c(t-\kappa) \equiv 0 \pmod{g}.
\end{equation}
Note that, although $ \kappa $ is only defined mod $\epsilon$, (\ref{cong-a}) and (\ref{cong-b})
make sense because of (\ref{X-ord-a}) and (\ref{X-ord-c}). We
determine when (\ref{cong-a}) and (\ref{cong-b}) hold mod $ q $ for each prime $q \mid e$.  

If $q \mid \zeta \delta$ then (\ref{cong-a}) and (\ref{cong-b}) are automatically
satisfied because of (\ref{X-ord-a}), (\ref{X-ord-c}). 

If $q \mid \gcd(\gamma,z)$ then (\ref{cong-b}) gives no condition mod $q$
and (\ref{cong-a}) reduces to $a(\kappa-t) \equiv 0$ since $k \equiv 1$. By
(\ref{X-reg-a}), this is the same as (\ref{com-sz}). 

Finally, suppose that $q \mid \gcd(\gamma,g)$. If $ t \equiv
\kappa $ then (\ref{cong-b}) holds for any choice of $ c $, and, substituting
into (\ref{cong-a}), we get (\ref{com-sg})(i). If $ t \not \equiv
\kappa $ then (\ref{cong-b}) gives $ c \equiv 0 $, and then (\ref{cong-a}) reduces to $tk
\equiv \kappa$, again using (\ref{X-reg-a}). This gives (\ref{com-sg})(ii).
\end{proof}

\begin{proposition} \label{reg-conds} 
Suppose that (\ref{XYkap-h}) and condition (i) of Lemma \ref{XY}
hold. Then condition (ii) of Lemma \ref{XY} is equivalent to
\begin{equation} \label{Adi-all}
 A(di) \mbox{ represents all residue classes mod } \gcd(\zeta\delta,e) \mbox{ as } i 
   \mbox{ varies.} 
\end{equation}
Moreover, this occurs if and only if
the following two conditions hold: 
\begin{equation} \label{s-z-ngamm}
  \mbox{ for } q \mid \gcd(\zeta \delta,z) \mbox{ we have }
           t \equiv 1 \mbox{ and } u \not \equiv 0;
\end{equation}
\begin{eqnarray} \label{s-g-ngamm}
  \lefteqn{\mbox{for } q \mid \gcd(\zeta \delta,g) \mbox{ we have either}  } \\
   & \mbox{(i)} &  t \equiv 1, \quad v \not  \equiv 0, \mbox{ or}  \nonumber \\
       & \mbox{(ii)} &  tk \equiv 1, \quad  v \not \equiv \mu u, \nonumber  
\end{eqnarray} 
where $\mu = k^{-1} z^{-1}(k-1) \in \Z_g^\times$.
\end{proposition}
\begin{proof}
By Lemma \ref{Y-pow}, we have $ Y^{di} \cdot 1_G= \sigma^{A(di)}$ for
$i \geq 0$. On the other hand, it follows from (\ref{pow-X}),
(\ref{X-reg-a}) and (\ref{X-ord-a}) that, for any $j \in \Z$, the
orbit of $\sigma^j$ under $\langle X \rangle$ is $
\{\sigma^{j+em/\gamma} \colon m \in \Z \}$. Thus condition (ii) of
Lemma \ref{XY} holds if and only if $A(di)$ represents all residue
classes mod $e/\gamma=\gcd(\zeta \delta,e)$ as $i$ varies. 
This proves the first assertion.

Suppose that (\ref{Adi-all}) holds. Then $A(di)$ represents all residue classes mod $q$ for each 
prime $q \mid \gcd(\zeta \delta,e)$. We analyse this condition,
distinguishing several cases. Recall that 
$$ A(di) = uS(tk,di) + vzkT(k,t,di) $$  

First let $q \mid \gcd(\zeta \delta,z)$. Then $k \equiv 1$, and we have 
$A(di)  \equiv u S(t,di)$. If $t \not \equiv 1$ then, by Proposition
\ref{ST-cong}(i),
$$ A(di) \equiv \frac{u(t^{di}-1)}{t-1}. $$
As there is no $i$ with $t^{di} \equiv 0$, this cannot take all values
mod $q$. If $t \equiv 1$, we have $A(di) \equiv udi$, which
takes all values mod $q$ if and only if $u \not \equiv
0$. This gives (\ref{s-z-ngamm}).

Now let $q \mid \gcd(\zeta \delta,g)$. Thus $k \not \equiv
1$, but of course $k^d \equiv 1$. If $t \not \equiv 1 \not \equiv tk$
then, as $(tk)^d \equiv t^d$, 
it follows from Proposition \ref{ST-cong} that each of the two terms of $A(di)$
is the product of $t^{di}-1$ by a term independent of $i$. Thus, as
above, $A(di)$ cannot represent all residue classes mod $q$.

If $q \mid \gcd(\zeta \delta,g)$ and $t \equiv 1$, we have  
$$ S(tk, di) \equiv S(k,di) \equiv \frac{k^{di}-1}{k-1} \equiv 0, $$
which by Proposition \ref{ST-cong}(ii) gives
$$ A(di) \equiv vzk T(k,1,di) \equiv \frac{vzdi}{k-1}. $$
This represents all residue classes mod $q$ if and only if $v \not
\equiv 0$, giving (\ref{s-g-ngamm})(i). 

If $q \mid \gcd(\zeta \delta,g)$ and $tk \equiv 1$ then
$$ A(di) \equiv udi + \frac{vzkdi}{k(t-1)} \equiv \frac{zk}{k-1} \,
(\mu u-v) di, $$ 
which represents all residue classes mod $q$ if and only if $v \not
\equiv \mu u$. This gives (\ref{s-g-ngamm})(ii).

We have now shown that if (\ref{Adi-all}) holds then (\ref{s-z-ngamm})
or (\ref{s-g-ngamm}) holds for each prime $q \mid \gcd(\zeta
\delta,e)$. Conversely, if (\ref{s-z-ngamm}) or one of the cases of
(\ref{s-g-ngamm}) holds for each such $q$, then it is clear from the
formula for $A(di)$ in each case that $A(di)$ represents all residue
classes mod $q$ as $i$ runs through any complete set of residues mod
$q$. It then follows from the Chinese Remainder Theorem that
(\ref{Adi-all}) holds.
\end{proof}

We extract the following information from the proof of Proposition
\ref{reg-conds}. 

\begin{corollary} \label{reg-conds-bis}
If (\ref{s-z-ngamm}) holds then $A(di) \equiv udi$.

If (\ref{s-g-ngamm})(i) holds then
$$      A(di) \equiv \frac{vzdi}{k-1}. $$

If (\ref{s-g-ngamm})(ii) holds then
$$ A(di) \equiv \frac{zk}{k-1} (\mu u-v)di. $$

Hence if $\Gamma^* \in \F_h$ then, for each $q \mid \gcd(\zeta
\delta,e)$, 
$$ q \mid A(di) \LRA q \mid i.  $$
\end{corollary}

\begin{proposition} \label{Y-ord}
Suppose that (\ref{XYkap-h}) and conditions (i) and (ii) of Lemma
\ref{XY} hold. Then condition (iii) of Lemma \ref{XY} is
equivalent to the two further conditions  
\begin{equation} \label{tk-kap}
  \mbox{if } q \in S_h^+ \cup T \mbox{ and } t \equiv \kappa k^{-1}
  \equiv 1 \mbox{ then } v  \equiv 0;
\end{equation} 
\begin{equation} \label{t-kap}
  \mbox{if } q \in S_h^- \cup T \mbox{ and } t \equiv \kappa \mbox{
    then } v \equiv \mu u.
\end{equation} 
Moreover, we then have 
\begin{equation} \label{Adi0}
   A(di) \equiv 0 \pmod{q} \mbox{ when } q \mid \gamma 
 \mbox{ and } i \equiv 0 \pmod{\gcd(\delta,e)}. \end{equation}
\end{proposition}
\begin{proof}
By Hypothesis \ref{HGS-exist}, we have $\zeta \delta=di_0$ where
$i_0=\gcd(\zeta \delta, e)$. It then follows from (\ref{com-sz}),
(\ref{com-sg}), (\ref{s-z-ngamm}) and (\ref{s-g-ngamm}) that $t^{di}
\equiv 1 \pmod{e}$ whenever $i_0 \mid i$.

By Lemma \ref{Y-pow}, condition (iii) of Lemma \ref{XY} is
equivalent to the following two congruences:  
\begin{equation} \label{Y-ord-1} 
    A(\zeta \delta) \equiv 0 \pmod{e};
\end{equation}
\begin{equation} \label{Y-ord-2}
     vS(t,\zeta \delta) \equiv 0 \pmod{g}.
\end{equation}
We will show that these are equivalent to (\ref{tk-kap}) and (\ref{t-kap}). 

We first check that (\ref{Y-ord-1}) and (\ref{Y-ord-2}) give no
conditions for primes $q \mid \gcd(\zeta \delta,e)=i_0$. 
We have $A(di_0) \equiv 0$ by Corollary
\ref{reg-conds-bis}, so (\ref{Y-ord-1}) holds.  If also $q \mid g$,
then by (\ref{s-g-ngamm}) either $t \equiv 1$, so that $S(t, \zeta
\delta)= \zeta \delta \equiv 0$, or $tk \equiv 1$, so that $S(t, \zeta
\delta) \equiv 0$ since $t^{\zeta \delta} \equiv 1$. Hence (\ref{Y-ord-2})
holds.

Next suppose that $q \mid \gcd(\gamma, z)$. Then (\ref{Y-ord-2}) gives
no condition at $q$. As (\ref{Y-ord-1}) is a special
case of (\ref{Adi0}), we just need to verify (\ref{Adi0}) for these
$q$. Now $t \equiv \kappa$ by (\ref{com-sz}), and also $k \equiv 1
\not \equiv \kappa$.  In this case,
$$ A(di) = u S(tk,di) \equiv 0 \mbox{ for all } i, $$
giving (\ref{Adi0}).

Finally, let $q \mid \gcd(\gamma,g)$. Recall that 
$$ A(di) = uS(tk,di) + vzkT(k,t,di). $$ 
If $t \not \equiv 1 \not \equiv tk$ then $S(tk, \zeta \delta) \equiv
T(k, t, \zeta \delta) \equiv S(tk, \zeta \delta) \equiv 0$, so
(\ref{Y-ord-1}), (\ref{Y-ord-2}) hold with no conditions on $u$,
$v$. Similarly, if $i_0 \mid i$ then (\ref{Adi0}) holds. It
remains to consider the special cases $t \equiv 1$ and $tk \equiv 1$.

If $t \equiv 1$, we cannot have $t \equiv \kappa$ since $\kappa
\not \equiv 1$, but
$$  tk \equiv \kappa \LRA k \equiv \kappa \LRA q \in S_h^+ \cup T. $$
In this case (\ref{Y-ord-2}) is equivalent
to $v \equiv 0$, and then (\ref{Y-ord-1}) holds for arbitrary $u$. This
gives (\ref{tk-kap}), and (\ref{Adi0}) also holds.

If $tk \equiv 1$, we cannot have $tk \equiv
\kappa$, but
$$ t \equiv \kappa \LRA  k \equiv \kappa^{-1} \LRA q \in S_h^- \cup T. $$ 
As $t \not \equiv 1$, (\ref{Y-ord-2}) holds for arbitrary $v$, and
(\ref{Y-ord-1}) becomes
$$ u \zeta \delta + zv\ \frac{k\zeta \delta}{k(t-1)} \equiv 0, $$
which simplifies to (\ref{t-kap}) since $t \equiv k^{-1}$. Again,
(\ref{Adi0}) holds.
\end{proof}

\begin{lemma} \label{quin}
A quintuple $(t,a,c,u,v) \in \Z_e^\times \times \Z_e \times \Z_g \times
  \Z_e \times \Z_g$ belongs to $\NN_h$ if and only if, for each prime
  $q \mid e$, its entries satisfy the conditions mod $q$ shown in
  Table \ref{quintuples}, where $\lambda$, $\mu$ are defined in
  Propositions \ref{XY-com}, \ref{reg-conds} respectively. 
\end{lemma}

\begin{remark}
Before proving Lemma \ref{quin}, we make some comments about how to
read Table \ref{quintuples}. All congruences are modulo the relevant
prime $q$. Entries shown as ``arb.'' may be chosen arbitrarily mod
$q$. Where there are two rows for a particular set of primes,
$(t,a,u,c,v)$ may satisfy the conditions in either row. For example,
if $q \in S_h^+$ then either $t \equiv 1$, $a \not \equiv 0$, $c
\equiv v \equiv 0$ or $t \equiv \kappa$, $a \not \equiv 0$, $c \equiv
\lambda a$. The cells for $c$, $v$ when $q \mid z$ are empty since
$c$, $v \in \Z_g$ so are not defined mod $q$. The final column of
Table \ref{quintuples} will be explained in the proof of Lemma
\ref{count-Nh} below.
\end{remark}

\begin{table}[ht] 
\centerline{ 
\begin{tabular}{|c|c|c|c|c|c|c|} \hline
  Primes $q$ & $t$ & $a$ & $u$ & $c$ & $v$ & Number  \\ \hline 
$ q \in P $ & $ \kappa $ & $ \not \equiv 0 $
  & arb.  & & & $q(q-1)$  \\ \hline  
$ q\in Q $ & $ 1 $ & $0 $ & $ \not \equiv 0$ & & & $q-1$ \\  \hline
 $q \in R \cup S_h'$ & $ \kappa $ & $  \not \equiv 0 $ & arb. 
              & $\lambda a$ & arb. & $2q^2(q-1)$ \\   
 & $\kappa k^{-1}$ & $ \not \equiv 0 $ &
                        arb. & $ 0 $ & arb. & \\ \hline
$q \in S_h^+$   & $ \kappa k^{-1} \equiv 1
      $ & $ \not \equiv 0 $ & arb. & $ 0 $ & $ 0 $ & $q(q^2-1)$ \\  
 & $\kappa $ & $\not \equiv 0 $ & arb. &   $\lambda a$   & arb. &  \\ \hline
$q \in S_h^-$ & $ \kappa $ & $ \not \equiv 0 $
                        & arb.  & $\lambda a$ & $\mu u$ & $q(q^2-1)$ \\  
	 & $\kappa k^{-1} \equiv \kappa^2$ & $\not \equiv 0$ & arb. &
                        $ 0 $ & arb. & \\ \hline  
 $q \in T$ & $\kappa \equiv -1$ & $\not \equiv 0$ 
                   & arb. & $\lambda a$ & $\mu u$ & $2q(q-1$) \\    
 & $\kappa k^{-1} \equiv 1$ & $ \not
                        \equiv0 $ & arb. & $0$  & $ 0 $ &\\ \hline   
$ q \in U$ & $1$ & $0$ & arb. & $ 0
                        $ & $ \not \equiv 0$ & $2q(q-1)$  \\  
	& $ k^{-1}$ & $0$ & arb. & $0$ & $\not \equiv \mu u$ &
                        \\ \hline 
\end{tabular}
}  
\vskip3mm

\caption{Conditions for membership of $\NN_h$.} 
 \label{quintuples}  	
\end{table}

\begin{proof}[Proof of Lemma \ref{quin}]

By Lemma \ref{XY} and Propositions \ref{X-ord}, \ref{XY-com},
\ref{reg-conds} and \ref{Y-ord}, $(t,a,c,u,v) \in \NN_h$ if and only
if the relevant conditions from (\ref{X-reg-a})--(\ref{com-sg}) and
(\ref{s-z-ngamm})--(\ref{t-kap}) hold for each prime $q \mid e$. We pick
out these conditions for each prime.

For $q \in P$, we have $q \mid \gcd(\gamma,z)$, so $a \not
\equiv 0$ by (\ref{X-reg-a}) and $t \equiv \kappa$ by (\ref{com-sz}).

For $q \in Q$, we have $q \mid \gcd(\zeta \delta,z)$, so
$a \equiv 0$ by (\ref{X-ord-a}) and $t \equiv 1$, $u \not \equiv 0$ by
(\ref{s-z-ngamm}). 

If $q \in R$ then $\kappa \not \equiv k^{ \pm 1}$ since $\kappa$
has order $\rho_q$ and $k^{\pm 1}$ has order $r_q \neq \rho_q$. If $q \in
S_h'$ then again $\kappa \not \equiv k^{ \pm 1}$. In both cases, 
$q \mid \gcd(\gamma,g)$. Thus by (\ref{com-sg}) either $t
\equiv \kappa$, $c \equiv \lambda a$ or $t \equiv \kappa k^{-1}$, $c
\equiv 0$. Moreover, $a \not \equiv 0$ by (\ref{X-reg-a}), and there
are no restrictions on $u$ and $v$.

For $q \in S_h^+$, we have $q \mid \gcd(\gamma,g)$ and
$\kappa \equiv k \not \equiv \pm 1$. Again, $a \not \equiv 0$ by
(\ref{X-reg-a}), and by (\ref{com-sg}) either $t \equiv \kappa$, $c
\equiv \lambda a$ or $t\equiv \kappa k^{-1} \equiv 1$, $c \equiv
0$. If $t \equiv 1$ then $v \equiv 0$ by (\ref{tk-kap}),
whereas if $t \equiv \kappa$ there is no restriction on $u$ and $v$.

For $q \in S_h^-$, we have $q \mid \gcd(\gamma,g)$ and $\kappa
\equiv k^{-1} \not \equiv \pm 1$. As before, $a \not \equiv 0$ and
either $t \equiv \kappa$, $c \equiv \lambda a$ or $t\equiv \kappa
k^{-1}$, $c \equiv 0$. If $t \equiv \kappa$ then $\kappa k \equiv 1$,
and by (\ref{t-kap}) we have $v = \mu u$, whereas if $t\equiv \kappa
k^{-1}$ there is no restriction on $u$ and $v$.

For $q \in T$, we have $q \mid \gcd(\gamma,g)$ and $\kappa \equiv
k \equiv -1 \not \equiv 1$. Again $a \not \equiv 0$ and either $t
\equiv \kappa \equiv -1$, $c \equiv \lambda a$ or $t\equiv \kappa
k^{-1} \equiv 1$, $c \equiv 0$. If $t
\equiv 1$ then $v \equiv 0$ by (\ref{tk-kap}), while if $t \equiv -1$
then $v \equiv \mu u$ by (\ref{t-kap}).

For $q \in U$ we have $q \mid \gcd(\zeta \delta,g)$. Then
$a \equiv 0$ and $c \equiv 0$ by (\ref{X-ord-a}) and
(\ref{X-ord-c}). Moreover, by (\ref{s-g-ngamm}), either $t \equiv 1$,
$v \not \equiv 0$ or $t \equiv k^{-1}$, $v \not \equiv \mu u$.
\end{proof}

\section{Counting Hopf-Galois Structures}

In this section, we complete the proof of Theorem \ref{thm-HGS}.

\begin{lemma} \label{count-Nh}
Suppose that Hypothesis \ref{HGS-exist} holds, and let $1 \leq h \leq
w$. Then
$$ |\NN_h| = \varphi(e) 2^{\omega(g)} g \gamma 
          \left(\prod_{q\in T} \frac{1}{q}\right) 
          \left(\prod_{q\in S_h}
              \frac{q+1}{2q}\right). $$
In particular, $\F_h$ is not empty.
\end{lemma}
\begin{proof}
By Lemma \ref{quin}, membership of $\NN_h$ is determined by conditions
mod $q$ for each prime $q \mid e$ separately, so, by the Chinese
Remainder Theorem, the number of possible quintuples in $\NN_h$ is the
product over all $q$ of the number of possible quintuples mod $q$. 

The final column of Table \ref{quintuples} above shows the number of
possibilities for $(t,a,c,u,v)$ mod $q$ for each prime $q$. (If $q
\mid z$, we ignore $c$, $v$ as they are only defined mod $g$.)  For
example, if $q \in S_h^+$ then either $t \equiv 1$ or $t\ \equiv
\kappa$ (where, as before, $\kappa$ really means $\kappa_h$).  If $t
\equiv 1$, we have $a \not \equiv 0$, $u$ arbitrary, but $c \equiv v
\equiv 0$, giving $(q-1)q$ quintuples mod $q$.  If $t \equiv \kappa$
then again $a \not \equiv 0$, $c \equiv \lambda a$, and now $u$, $v$
are arbitrary. This gives a further $(q-1)q^2$ quintuples.  Thus the
total number of quintuples mod $q$ is
$(q-1)q+(q-1)q^2=q(q^2-1)$, as shown in Table \ref{quintuples}. The
other entries in the final column of Table \ref{quintuples} are
obtained by similar (but usually simpler) calculations, which we leave
to the reader.

Noting that $\prod_{q \mid e} (q-1)= \varphi(e)$, we therefore have 
\begin{eqnarray*}
 |\NN_h| & = & \varphi(e) 
              \left(\prod_{q\in P} q \right)
               \left(\prod_{q\in R \cup S_h'} 2q^2 \right) 
          \left(\prod_{q\in S_h}
              q(q+1) \right)
          \left(\prod_{q\in T \cup U} 2q\right) \\
  & = & \varphi(e) \left(\prod_{q\in P} q \right) 
         \left(\prod_{q\in R \cup S \cup T} 2q^2 \right) 
          \left(\prod_{q\in S_h}
              \frac{q+1}{2q}\right)
          \left(\prod_{q\in T} \frac{1}{q}\right) 
          \left(\prod_{q\in U} 2q \right) \\
  & = & \varphi(e) 
         \left(\prod_{q\in R \cup S \cup T \cup U} 2q \right) 
         \left(\prod_{q\in P \cup R \cup S \cup T } q \right) 
          \left(\prod_{q\in T} \frac{1}{q}\right) 
          \left(\prod_{q\in S_h}
              \frac{q+1}{2q}\right) \\
  & = & \varphi(e) 
         \left(\prod_{q \mid g} 2q \right) 
         \left(\prod_{q\mid \gamma} q \right) 
          \left(\prod_{q\in T} \frac{1}{q}\right) 
          \left(\prod_{q\in S_h}
              \frac{q+1}{2q}\right) \\
  & = & \varphi(e) (2^{\omega(g)} g) \gamma 
          \left(\prod_{q\in T} \frac{1}{q}\right) 
          \left(\prod_{q\in S_h}
              \frac{q+1}{2q}\right). \\
\end{eqnarray*}
\end{proof}

\begin{proof}[Proof of Theorem \ref{thm-HGS}]
Given the two groups $G$, $\Gamma$ of squarefree order $n$, we want to
determine the number $e(\Gamma,G)$ of Hopf-Galois structures of type
$G$ on a Galois extension with Galois group isomorphic to
$\Gamma$. This is related by the formula (\ref{HGS-count-formula}) to
the number $e'(G,\Gamma)$ of regular subgroups in $\Hol(G)$ isomorphic
to $\Gamma$.

If $\gamma \nmid e$ then by Proposition \ref{d-div-gam-del} there are
no such regular subgroups and hence no Hopf-Galois structures.

If $\gamma \mid e$, so Hypothesis \ref{HGS-exist} holds, then by
(\ref{sum-Nh}) and Lemma \ref{count-Nh} we have
$$ e'(G,\Gamma)  = 
  \frac{\varphi(\delta)  2^{\omega(g)}g}{w} 
          \left(\prod_{q\in T} \frac{1}{q}\right)
           \sum_{h=1}^w  \left(\prod_{q\in S_h}
              \frac{q+1}{2q}\right).         $$
Since $|\Aut(G)|=g \varphi(e)$ by Lemma \ref{Aut-G}, and similarly
$|\Aut(\Gamma)|=\gamma \varphi(\epsilon)$, it follows from
(\ref{HGS-count-formula}) that 
\begin{eqnarray*}
    e(\Gamma,G) & = & 
        \frac{|\Aut(\Gamma)|}{|\Aut(G)|} e'(G,\Gamma) \\
    & = & \frac{\gamma \varphi(\epsilon) \varphi(\delta) 
          2^{\omega(g)} g}{g \varphi(e)w} 
          \left(\prod_{q\in T} \frac{1}{q}\right)
           \sum_{h=1}^w  \left(\prod_{q\in S_h}
              \frac{q+1}{2q}\right).         
\end{eqnarray*}
As $\varphi(\epsilon)
\varphi(\delta)=\varphi(n)=\varphi(e)\varphi(d)$, we therefore have
$$     e(\Gamma,G) = 
 \frac{\gamma \varphi(d) 2^{\omega(g)}}{w} 
          \left(\prod_{q\in T} \frac{1}{q}\right)
           \sum_{h=1}^w  \left(\prod_{q\in S_h}
              \frac{q+1}{2q}\right).  $$
This completes the proof of Theorem \ref{thm-HGS}.
\end{proof}

\section{An example where $n$ has $7$ prime factors} \label{example}

We consider some of the groups of order 
$$ n = 2 \cdot 3 \cdot 7 \cdot 43 \cdot 127 \cdot 211 \cdot 337
    = 16\;309\;243\;734, $$ 
a squarefree integer with $7$ prime factors. The numerical
calculations in this section were performed using MAPLE. 

Using a formula of H\"older \cite{holder}, recalled as
\cite[eqn.~(1)]{AB}, we find that there are $272\;736$ isomorphism
classes of groups of order $n$.  We consider just four of these, all
of which have $g=43 \cdot 127 \cdot 211 \cdot 337$. 

We first choose $a_1\ \in \Z_{43}^\times$, $a_2 \in \Z_{127}^\times$, $a_3
\in \Z_{211}^\times$, $a_4 \in \Z_{337}^\times$ all having order
$42$. (For example, we may take $a_1=3$, $a_2=5$, $a_3=26$,
$a_4=21$.)  We then specify our groups $G$ by stipulating that $k$ is
congruent to a certain power of $a_1$ (resp.~$a_2$, $a_3$, $a_4$)
mod $43$ (resp.~$127$, $211$, $377$) as shown in Table
\ref{example-table}. 

\begin{table}[ht] 
\centerline{ 
\begin{tabular}{|c|c|c|c|c|c|c|c|c|c|} \hline
 & $k \bmod 43$ & $k \bmod 127$ & $k \bmod 211$ & $k \bmod 337$
 & $r_{43}$ & $r_{127}$ & $r_{211}$ & $r_{337}$ & $d$ \\ \hline
 $G_1$ & $a_1^{21}$ & $a_2^{14}$  & $a_3^{6}$ & $a_4^{2}$ & 
  $2$ & $3$ & $7$ & $21$ & $42$ \\ \hline
 $G_2$ & $a_1^{21}$ & $a_2^{14}$  & $a_3^{12}$ & $a_4^{2}$ & 
  $2$ & $3$ & $7$ & $21$ & $42$ \\ \hline
 $G_3$ & $a_1$ & $a_2^{2}$ &  $a_3^3$ & $a_4^{6}$ & 
  $42$ & $21$ & $14$ & $7$ & $42$ \\ \hline
 $G_4$ & $a_1^{21}$ & $a_2^{6}$ & $a_3^{6}$ & $a_4^{3}$ & 
  $2$ & $7$ & $7$ & $14$ & $14$ \\ \hline
\end{tabular}
}  
\vskip3mm

\caption{Parameters for some groups of order $n$.} 
 \label{example-table}  	
\end{table}

We also show in Table \ref{example-table} the value of $r_q$ for each
of the primes $q$ dividing $g$. In each case, $d$ is the least common
multiple of these values. For $G_1$, $G_2$, $G_3$ we have $d=42$ so
$z=1$ and $e=g$. For $G_4$, we assume further that $k \equiv 1 \pmod{3}$, so
that $d=14$ and $z=3$ and $e=3g$.

\begin{remark}  \label{non-isom}
No two of our groups $G_i$ are isomorphic, since no two of our choices
of $k$ generate the same cyclic subgroup of
$\Z_e^\times$. Nevertheless, $G_1$ and $G_2$ have the same values of
$d$, $g$, $z$ and $r_q$ for all $q \mid e$, showing that these
parameters are not in general sufficient to determine the isomorphism
class of $G$. Moreover, $d$, $g$ and $z$ are not in general sufficient
to determine the $r_q$, as shown by $G_2$ and $G_3$.
\end{remark}

We now fix $\Gamma=G_1$, without loss of generality taking $\kappa$ to
be the value of $k$ specified for $G_1$. We will use Theorem
\ref{thm-HGS} to determine $e(\Gamma,G_i)$ for $i=1, \ldots, 4$.

\subsection{$e(\Gamma,G_1)$:}  \label{G-is-Gam}

Since $d=\delta$, the group $\Delta$ is trivial and
$w=\varphi(d)=12$. The set $\K$ consists of the $12$ elements
$\kappa$, $\kappa^5$, $\kappa^{11}$, $\kappa^{13}$, $\kappa^{17}$,
$\kappa^{19}$, $\kappa^{23}$, $\kappa^{25}$, $\kappa^{29}$,
$\kappa^{31}$, $\kappa^{37}$, $\kappa^{41}$. We label these as
$\kappa_1, \ldots, \kappa_{12}$ in the order listed. Then the
$\kappa_h$ are chosen as in Remark \ref{choose-kappa}.

We have $R=\emptyset$, $S=\{127, 211, 337\}$, $T=\{43\}$.  
We determine the sets $S_h^+$ by considering each prime $q \in S$. For
$q=127$, we have $r_q=3$ so $k \equiv \kappa_h$ if and only if
$\kappa_h \equiv \kappa^r$ with $r \equiv 1 \pmod{3}$. This occurs for
$h=1$, $4$, $6$,
$8$, $10$, $11$. For $q=211$, we have $r_q=7$ and $k \equiv \kappa_h$
for $h=1$, $9$. For $q=337$, we have $r_q=21$ and $k \equiv \kappa_h$
only for $h=1$. Thus 
$$ S_h^+ = \begin{cases} \{127, 211, 337\} & \mbox{if } h=1, \\
           \{127\} & \mbox{if } h=4, 6, 8, 10, 11, \\
           \{211\} & \mbox{if } h=9,\\
         \emptyset & \mbox{if } h= 2, 3, 5, 7, 12. \end{cases} $$
For each $q \in S$, we have $q \in S_h^+$ for $w/\varphi(r_q)$ values
of $h$, as explained in Proposition \ref{num-h}. Since
$S_h^-=S_{w+1-h}^+$, the sets
$S_h=S_h^+ \cup S_h^-$ are as follows:
$$ S_h = \begin{cases} \{127, 211, 337\} & \mbox{if } h=1, 12 \\
           \{127\} & \mbox{if } h=2, 3, 5, 6, 7, 8, 10, 11, \\
           \{127, 211\} & \mbox{if } h=4, 9. \end{cases} $$

Then the number $e(\Gamma,G_1)=e(\Gamma, \Gamma)$ of Hopf-Galois structures is
\begin{eqnarray*}
e(\Gamma,G_1) 
   & = & \frac{2^{\omega(g)} \varphi(d) \gamma}{w} 
       \left( \prod_{q \in T} \frac{1}{q} \right) \sum_{h=1}^w
       \prod_{q \in S_h} \frac{q+1}{2q} \\
 & = & \frac{2^4 \varphi(42) \cdot 43 \cdot 127 \cdot 211 \cdot
         337}{12} \cdot \frac{1}{43} \cdot \\
 & & \qquad 
    \left(2 \times \frac{128}{254} 
       \times \frac{212}{422} \times \frac{338}{674} + 
      8 \times \frac{128}{254} + 2 \times \frac{128}{254} 
       \times  \frac{212}{422} \right) \\
     & = & 692\;355\;072.
\end{eqnarray*}

\subsection{$e(\Gamma,G_2)$:}

We now have $211 \in S_h$ if and only if $\kappa_h \equiv \kappa^r$
with $r \equiv \pm 2 \pmod{7}$, which occurs for $h=2$, $6$, $7$,
$11$. Thus 
$$ S_h = \begin{cases} \{127, 337\} & \mbox{if } h=1, 12 \\ \{127,
  211\} & \mbox{if } h=2, 6, 7, 11, \\ \{127\} & \mbox{if } h=3, 4, 5, 8, 9,
  10. \end{cases} $$ 
The sets $S_h$ in this case are different from those in \S8.1, even
though the $r_q$ have not changed.  The number of Hopf-Galois
structures is now
\begin{eqnarray*}
e(\Gamma,G_2) 
 & = & \frac{2^4 \varphi(42) \cdot 43 \cdot 127 \cdot 211 \cdot
         337}{12} \cdot \frac{1}{43} \cdot \\
 & &  
    \left(2 \times \frac{128}{254} \times \frac{338}{674} + 
    4 \times \frac{128}{254} \times \frac{212}{422} + 
    6 \times \frac{128}{254} \right) \\
     & = & 656\;228\;352.
\end{eqnarray*}

\subsection{$e(\Gamma,G_3)$:}
This time there are no primes $q \mid e$ with $r_q = \rho_q$, so 
$S=T=\emptyset$ and $R=\{43, 127, 211, 337\}$. 
The number of Hopf-Galois structures is now
\begin{eqnarray*}
e(\Gamma,G_3) 
 & = & \frac{2^4 \varphi(42) \cdot 43 \cdot 127 \cdot 211 \cdot
         337}{12}  \cdot \sum_{h=1}^{12} 1 \\
     & = & 74\;556\;542\;784.
\end{eqnarray*}

\subsection{$e(\Gamma,G_4)$:}

In the previous cases, we had $w=|\K|$, so that the orbits of $\Delta$
on $\K$ were singletons. In our final case, this does not hold.
Indeed, we have $w=\varphi(\gcd(14,42))=6$, and a system of orbit
representatives of $\K$ under $\Delta=\{1,29\}$, chosen as in Remark
\ref{choose-kappa}, is $\kappa_1=\kappa$, $\kappa_2=\kappa^5$,
$\kappa_3=\kappa^{11}$, $\kappa_4=\kappa^{31}$,
$\kappa_5=\kappa^{37}$, $\kappa_6=\kappa^{41}$.

We find $R=\{127, 337\}$, $S=\{211\}$, $T=\{43\}$. As $k \equiv \kappa
\pmod{211}$ and $r_{221}=7$, we have $211 \in S_h^+$ only for $h=1$
and $211 \in S_h^-$ only for $h=6$.  The number of Hopf-Galois
structures is now
\begin{eqnarray*}
e(\Gamma,G_4) 
 & = & \frac{2^4 \varphi(14) \cdot 43 \cdot 127 \cdot 211 \cdot
         337}{6} \cdot \frac{1}{43} \cdot \\
 &  & \qquad \left( 2 \times \frac{212}{422} + 4 \times 1 \right) \\
     & = & 723\;131\;904.
\end{eqnarray*}

\section{Special Cases}  \label{special}

In this section, we give some corollaries of our main result Theorem 
\ref{thm-HGS}. 

\subsection{When $\Gamma$ or $G$ is cyclic or dihedral}

We begin by reproving \cite[Theorem 1]{AB}

\begin{corollary} \label{Gam-cyc}
Let $G=G(d,e,k)$ be an arbitrary group of squarefree order $n$. Then 
a cyclic extension of degree $n$ admits precisely 
$2^{\omega(g)} \varphi(d)$. Hopf-Galois structures of type $G$, where
$g=e/\gcd(k-1,e)$.  
\end{corollary}
\begin{proof}
For $\Gamma$ cyclic, we have $\gamma=\delta=1$, so $w=1$ and the sets
$R$, $S$, $T$ are empty. Then Theorem \ref{thm-HGS} gives
$$ e(\Gamma,G) = 2^{\omega(g)} \varphi(d). $$
\end{proof}

\begin{corollary} \label{G-cyc}
Let $\Gamma=G(\delta,\epsilon,\kappa)$ be an arbitrary group of
squarefree order
$n$. A Galois extension with group $\Gamma$ admits precisely $\gamma=
\epsilon/\gcd(\kappa-1,\epsilon)$ Hopf-Galois structures of cyclic
type.
\end{corollary} 
\begin{proof}
For $G$ cyclic, we have $d=g=1$ so again $w=1$ and $R$, $S$, $T$ are
empty. Thus $e(\Gamma,G)= \gamma$.
\end{proof}

\begin{corollary} \label{Gam-dih}
Let $m$ be an odd squarefree integer, let $L/K$ be a Galois
extension of degree $n=2m$ with $\Gal(L/K) \cong D_{2m}$, the dihedral group
of order $n$. Let $G=G(d,e,k)$ be
an arbitrary group of 
order $n$. Then, with $g=e/\gcd(k-1,e)$, the number of Hopf-Galois
structures of type $G$ on $L/K$ is 
$$ e(D_{2m},G) = \begin{cases} m & \mbox{ if } d=1, \\
    \displaystyle{\frac{2^{\omega(g)}m}{g}} & \mbox{ if } d=2, \\
     0 & \mbox{ otherwise.} \end{cases} $$
\end{corollary}
\begin{proof}
Taking $\Gamma=D_{2m}$, we have $\delta=2$, $\gamma=m$, $\zeta=1$,
$w=1$ and $\rho_q=2$ for all primes $q \mid m$. If $d=1$ the result
follows from Corollary \ref{G-cyc}. If $d>2$ then $\gamma \nmid e$, so
$e(D_{2m},G)=0$. If $d=2$, we have $r_q=2$ for each prime $q \mid g$,
so $T=\{ q \mid g\}$ and $R$, $S$ are empty. The stated formula
follows from Theorem \ref{thm-HGS}.
\end{proof}

\begin{corollary} \label{G-dih}
Let $m$ be an
odd squarefree integer, and let $\Gamma=G(\delta,\epsilon,\kappa)$ be an
arbitrary group of order $n=2m$. Then the number of Hopf-Galois
structures of dihedral type on a Galois extension with Galois group $\Gamma$ is
$$ e(\Gamma,D_{2m}) = 2^{\omega(m)} \gamma \left( \prod_{q \mid \gamma, \; 
  \rho_q=2} \frac{1}{q} \right). $$
\end{corollary}
\begin{proof}
We let $G=D_{2m}$, so $d=2$, $g=m$, $z=1$ and $w=1$. Since $2 \nmid
\gamma$ (see Remark \ref{g-odd}) we necessarily have $\gamma \mid g$. 
Then $r_q=2$ for all $q \mid m$, so $R$ and $S$ are empty and $T=\{ q
\mid \gamma : \rho_q=2\}$. The stated formula then follows from
Theorem \ref{thm-HGS}.
\end{proof}

\subsection{When $n$ is the product of two primes}

We next recover the results of \cite{pq} which count Hopf-Galois
structures on Galois extensions of degree $n=pq$, where $p>q$ are
primes. We assume that $p \equiv 1\pmod{q}$ since otherwise any group
of order $pq$ is cyclic. Thus we have two groups of order $n$, the
cyclic group $C_n$ (for which $g=d=1$ and $z=pq$) and the nonabelian
group $C_p \rtimes C_q$ (for which $g=p$, $d=q$, $z=1$).

\begin{corollary} 
Let $n=pq$, where $p$, $q$ are primes with $p \equiv 1 \pmod{q}$. Let
$\Gamma$, $G$ be groups of order $pq$. Then the number of Hopf-Galois
structures of type $G$ on a Galois extension with Galois group
$\Gamma$ is as given in Table \ref{pq-table}.
\end{corollary}

\begin{table}[ht] 
\centerline{ 
\begin{tabular}{|c|c|c|} \hline
         & $G=C_n$ & $G=C_p \rtimes C_q$ \\ \hline
 $\Gamma=C_n$ & $1$   &  $2(q-1)$    \\ \hline
 $\Gamma=C_p \rtimes C_q$ &  $p$  & $2p(q-2)+2$  \\ \hline
\end{tabular}
}  
\vskip3mm

\caption{Hopf-Galois structures for two primes.} 
 \label{pq-table}  	
\end{table}

\begin{proof}
If either $G$ or $\Gamma$ is cyclic, the result follows from Corollary
\ref {Gam-cyc} or Corollary \ref{G-cyc}. If $G=\Gamma=C_p \rtimes C_q$
with $q=2$ (so $G$ and $\Gamma$ are both dihedral), the result follows from
Corollary \ref{Gam-dih}, or, alternatively, from Corollary
\ref{G-dih}. It remains to consider the case $G=\Gamma=C_p \rtimes
C_q$ with $q>2$. Then $d=\delta=q$ so $w=\varphi(q)=q-1$. Thus $R$
and $T$ are empty and $S=\{p\}$ with $r_p=\rho_p=q-1$. There is one value of
$h$ with $p \in S_h^+$ and one with $p \in S_h^-$. We therefore have
$$ e(\Gamma,G) = 2p \left( 2 \times \frac{p+1}{2p} + (q-3) \right)
   =  2p(q-2)+2. $$
\end{proof}

\subsection{When $n$ is the product of three primes}

In this final subsection, we consider the case where $n$ is the
product of three distinct primes $p_1<p_2<p_3$.  
This extends work of
Kohl \cite{kohl13, kohl16}.  In \cite{kohl13}, Kohl takes $p_1=2$ and
$p_3=2p_2+1$, so that $q=p_2>2$ is a Sophie Germain prime and $p=p_3$
is a safeprime. He calculates $e(\Gamma,G)$ for all possible pairs
$\Gamma$, $G$. The results for $\Gamma=C_p \rtimes C_{p-1} =
\Hol(C_p)$ were already obtained in \cite{Ch03}.  In
\cite{kohl16}, Kohl handles the case $p_1>2$,
with $p_3 \equiv p_2 \equiv 1 \pmod{p_1}$ but $p_3 \not \equiv 1
\pmod{p_2}$. 

We first describe the groups of order $n= p_1 p_2 p_3$.  Subject to
certain congruence conditions, there are $6$ possible factorisations
$n=dgz$ which give rise to groups of order $n$. For ease of reference,
we label these factorisations $1$--$6$ as shown in Table
\ref{3-prime-table}. The last column shows the number of isomorphism
types of group $G$ for each factorisation. 

\begin{table} 
\centerline{ 
\begin{tabular}{|c|ccc|c|c|} \hline
 Factorisation & $d$ & $g$ & $z$ & Condition & \# groups  \\ \hline
 $1$ &  $1$ & $1$ & $p_1 p_2 p_3$ & & $1$  \\
 $2$ &  $p_1$ & $p_2$ & $p_3$ & $p_2 \equiv 1 \pmod{p_1}$ & $1$  \\
 $3$ & $p_1$ & $p_3$ & $p_2$ & $p_3 \equiv 1 \pmod{p_1}$ & $1$  \\
 $4$ & $p_1$ & $p_2 p_3$ & $1$ & $p_2 \equiv p_3 \equiv 1 \pmod{p_1}$ &
 $p_1-1$  \\ 
 $5$ & $p_2$ & $p_3$ & $p_1$ & $p_3 \equiv 1 \pmod{p_2}$ & $1$  \\ 
 $6$ & $p_1 p_2$ & $p_3$ & $1$ & $p_3  \equiv 1 \pmod{p_1 p_2}$ &
 $1$  \\  \hline
\end{tabular}
}  
\vskip5mm

\caption{Isomorphism types for groups of order 
  $n=p_1 p_2 p_3$.}  \label{3-prime-table} 
\end{table}   

In all cases except Factorisation $4$, the group $\Z_g^\times$ is
cyclic and therefore has a unique subgroup of order $d$, so that
there is a unique isomorphism type group for the given
factorisation. For Factorisation $4$, however, the Sylow
$p_1$-subgroup of $\Z_g^\times = \Z_{p_2}^\times \times
\Z_{p_3}^\times$ is of order $p_1^2$ and contains $p_1+1$ subgroups of
order $d=p_1$, of which $p_1-1$ project nontrivally to both
factors. In this case, we obtain $p_1-1$ distinct
isomorphism classes. Thus there are in total $p_1+4$
isomorphism classes of groups of order $n$, provided that all the
indicated congruence conditions hold.

We examine more carefully the $p_1-1$ groups from Factorisation $4$ when
$p_1>2$. We may assume that the corresponding values of $k$ are all
congruent mod $p_2$, but run through all $p_1-1$ residue classes of order $p_1$
mod $p_3$. Given one such group $G=G(p_1,p_2 p_3,k)$, we write
$\widehat{G}=G(p_1,p_2 p_3, \widehat{k})$ where  
$\widehat{k} \equiv k \pmod{p_2}$ and $\widehat{k} \equiv k^{-1}
\pmod{p_3}$. Thus the groups from Factorisation $4$ come in pairs $G$,
$\widehat{G}$. 

We now determine $e(\Gamma,G)$ as the factorisations of $n$
corresponding to $G$ and $\Gamma$ each run through
the $6$ possibilities in Table \ref{3-prime-table}. 
For brevity, we state
our result only in the situation where all the congruence conditions
$p_i \equiv 1 \pmod{p_j}$ for $i>j$ are satisfied. The results in the
other cases are easily obtained from this, merely by omitting those
$G$ and $\Gamma$ for which the relevant congruence conditions (as
shown in Table \ref{3-prime-table}) do not hold.

\begin{theorem} \label{3-prime-thm}
Let $n=p_1 p_2 p_3$ where $p_1 < p_2 < p_3$ are primes and $p_i \equiv
1 \pmod{p_j}$ for $i>j$. Let $G$, $\Gamma$ be groups of order
$n$. Then $e(\Gamma,G)$, the number of Hopf-Galois structures of type
$G$ on a Galois extension $L/K$ with $\Gal(L/K)\cong \Gamma$, is as shown in
Table \ref{HGS-2} if $p_1=2$, respectively Table \ref{HGS-not2} if
$p_1>2$.  The rows (respectively, columns) of these tables correspond
to the factorisation of $n$ giving $\Gamma$ (respectively,
$G)$, as in Table \ref{3-prime-table}. In the case that $p_1>2$
and $G$, $\Gamma$ both come from Factorisation $4$, we have 
\begin{equation} \label{lastcase}
 e(\Gamma,G) = \begin{cases}
  4p_1 p_2 p_3 -10p_2 p_3 +2p_2 + 2p_3 + 2 & \mbox{ if } G \cong \Gamma
  \mbox{ or } \widehat{\Gamma}, \\
   4p_1 p_2 p_3-12 p_2 p_3 +4p_2 + 4p_3
  & \mbox{ if } G \not \cong \Gamma  \mbox{ or } \widehat{\Gamma}. 
 \end{cases}
\end{equation}
\end{theorem}

\begin{table} 
\centerline{ 
\begin{tabular}{|c|c|c|c|c|c|c|} \hline
$\downarrow \Gamma \quad G \rightarrow$  &$1$ & $2$ & $3$
  & $4$ & $5$ & $6$ \\ \hline 
$1$
   & $1$ & $2$ & $2$ & $4$ & $2(p_2-1)$ & $2(p_2-1)$ \\ \hline
$2$
   & $p_2$ & $2$ & $2p_2$ & $4$ & $0$ & $0$ \\ \hline
$3$ 
   & $p_3$ & $2p_3$ & $2$ & $4$ & $2(p_2-1)p_3$ & $2(p_2-1)p_3$  \\ \hline
$4$ 
   & $p_2 p_3$ & $2p_3$ & $2p_2$ & $4$ & $0$ & $0$ \\ \hline 
$5$
   & $p_3$ & $2p_3$ & $2p_3$ & $4p_3$ & $2+2(p_2-2)p_3$ & $2(p_2-1)p_3$ \\
  \hline 
$6$
   & $p_3$ & $2p_3$ & $2 p_3$ & $4p_3$ & $2(p_2-1)p_3$ &
  $2+2(p_2-2)p_3$  \\  \hline
\end{tabular}
}  
\vskip5mm

\caption{Numbers of Hopf-Galois structures for
  $n=p_1 p_2 p_3$ with $p_1=2$.}  \label{HGS-2} 
\end{table}

\begin{sidewaystable} 

\vskip8cm  

\centerline{ 
\begin{tabular}{|c|c|c|c|c|c|c|} \hline
$\downarrow \Gamma \quad G \rightarrow$   &$1$ & $2$ & $3$
  & $4$ & $5$ & $6$ \\ \hline 
$1$
   & $1$ & $2(p_1-1)$ & $2(p_1-1)$ & $4(p_1-1)$ & $2(p_2-1)$ &
  $2(p_1-1)(p_2-1) $ \\ \hline
$2$
   & $p_2$ & $2+2(p_1-2)p_2$ & $2(p_1-1)p_2$ & $4+4(p_1-2)p_2$ &
   $0$ & $0$ \\ \hline
$3$ 
   & $p_3$ & $2(p_1-1)p_3$ & $2+2(p_1-2)p_3$ &
  $4+4(p_1-2)p_3$ & $2(p_2-1)p_3$ & $2(p_1-1)(p_2-1)p_3$  \\ \hline
$4$ 
   & $p_2 p_3$ & $2p_3+2(p_1-2)p_2p_3$ & $2p_2+2(p_1-2)p_2p_3$ & 
    See (\ref{lastcase}) & $0$ & $0$ \\ \hline 
$5$
   & $p_3$ & $2(p_1-1)p_3$ & $2(p_1-1)p_3$ & $4(p_1-1)p_3$ &
  $2+2(p_2-1)p_3$ & $2(p_1-1)(p_2-1)p_3$ \\ 
  \hline 
$6$
   & $p_3$ & $2(p_1-1)p_3$ & $2(p_1-1)p_3$ & $4(p_1-1)p_3$ &
  $2(p_2-1)p_3$ & $2+2(p_1p_2-p_1-p_2)p_3$  \\  \hline
\end{tabular}
}  
\vskip5mm

\caption{Numbers of Hopf-Galois structures for
  $n=p_1 p_2 p_3$ with $p_1>2$.}  \label{HGS-not2}

\bigskip \bigskip
 
\centerline{ 
\begin{tabular}{|c|c|c|c|c|c|c|} \hline
$\downarrow \Gamma \quad G \rightarrow$  &$1$ & $2$ & $3$
  & $4$ & $5$ & $6$ \\ \hline 
$1$
   & $w=1$ & $w=1$ & $w=1$ & $w=1$ & $w=1$ & $w=1$ \\ \hline
$2$
   & $w=1$ & $w=p_1-1$ & $w=p_1-1$ & $w=p_1-1$ &  &  \\
 &  & $T=\{p_2\}$ & &  $T=\{p_2\}$ & &  \\  \hline
$3$ 
   & $w=1$ & $w=p_1-1$ & $w=p_1-1$ & $w=p_1-1$ & $w=1$ & $w=p_1-1$ \\
 & & & $T=\{p_3\}$ & $T=\{p_3\}$ & $R=\{p_3\}$ &  $R=\{p_3\}$ \\ \hline
$4$ 
   & $w=1$ & $w=p_1-1$ & $w=p_1-1$ & $w=p_1-1$ &  &  \\
 & & $T=\{p_2\}$ & $T=\{p_3\}$ & $T=\{p_2,p_3\}$ & &  \\ \hline 
$5$
   & $w=1$ & $w=1$ & $w=1$ & $w=1$ & $w=p_2-1$ & $w=p_2-1$ \\
 & &  &  $R=\{p_3\} $ &  & $S=\{p_3\}$ &  $R=\{p_3\}$ \\ \hline 
$6$
   & $w=1$ & $w=p_1-1$ & $w=p_1-1$ & $w=p_1-1$ & $w=p_2-1$ &
  $w=(p_1-1)(p_2-1)$  \\ 
 & &  & $R=\{p_3\}$ & $R=\{p_3\}$ & $R=\{p_3\}$ & $S=\{p_3\}$ \\ \hline
\end{tabular}
}  
\vskip5mm

\caption{$w$, $R$, $S$ and $T$ for $n=p_1 p_2 p_3$ with $p_1=2$.}  \label{wRST} 
\end{sidewaystable}   

\begin{proof} 
The entries shown as $0$ in either table correspond to cases where
$\gamma \nmid e$, so by Proposition \ref{d-div-gam-del} there are no
regular subgroups in $\Hol(G)$ isomorphic to $\Gamma$ and
$e(\Gamma,G)=0$. 

We show in Table \ref{wRST} the number $w$ and the sets $R$, $S$, $T$
for each combination of $G$ and $\Gamma$. For brevity, we only include
the sets $R$, $S$, $T$ when they are nonempty, and we give these sets
in the case $p_1=2$.  If $p_1 \neq 2$, we must replace $S$ by $S \cup
T$ and $T$ by the empty set (cf.~Remark \ref{Sh-wd}).  The four empty
cells correspond to cases where Hypothesis \ref{HGS-exist} is not
satisfied.

Except in the case $p_1>2$ and $g=\gamma=p_2 p_3$ (that is, $G$ and
$\Gamma$ both come from Factorisation $4$), the set $S$ is either empty
or consists of a single prime $p$. We have
\begin{equation} \label{S-empty}
    e(\Gamma,G) = 2^{\omega(g)} \varphi(d) \gamma \left(
\prod_{q \in T} \frac{1}{q} \right) \mbox{ if } S=\emptyset. 
\end{equation}
If $S=\{p\}$ then  $p=p_2$ or $p_3$, and $r_p=d>2$. Then
$p \in S_h^+$ (respectively, $S_h^-$) for $w/\varphi(d)$
values of $h$, so the sum in the expression
for the number $e(\Gamma,G)$ of 
Hopf-Galois structures simplifies to
\begin{eqnarray*}
  \sum_{h=1}^w \prod_{q \in S_h} \frac{q+1}{2q} & = & 
      \frac{2w}{\varphi(d)} \cdot \frac{p+1}{2p} + \left(w -
      \frac{2w}{\varphi(d)} \right) \\
  & = & \frac{w}{\varphi(d)p} \left( 1+[\varphi(d)-1]p\right).
\end{eqnarray*}
Then we have 
\begin{equation} \label{S-p} 
  e(\Gamma,G) = \frac{2^{\omega(g)}\gamma}{p} \left(
\prod_{q \in T} \frac{1}{q} \right) 
   \left(1+[\varphi(d)-1] p\right) \mbox{ if } S=\{p\}.
\end{equation}
From Table \ref{wRST}, (\ref{S-empty}) and (\ref{S-p}), we find that
that when $p_1=2$ the values of $e(\Gamma,G)$ are as in Table
\ref{HGS-2}, and when $p_1>2$ (but $G$ and $\Gamma$ do not both come
from Factorisation $4$), they are as in Table \ref{HGS-not2}.

We now examine the omitted case where $p_1>2$ and $G$, $\Gamma$ both
come from Factorisation $4$. We have $d=\delta=p_1$ and $w=p_1-1 \geq
2$. If $\kappa \equiv k$ or $\widehat{k}$ then $S_h=\{p_2,p_3\}$ for
two of the $p_1-1$ possible 
values of $h$, and $S_h=\emptyset$ for the rest. This occurs for $G
\cong \Gamma$ or $\widehat{\Gamma}$. Otherwise, there are two values
of $h$ with $S_h=\{p_2\}$ and a further two with $S_h=\{p_3\}$.  Thus,
for $G \cong \Gamma$ or $\widehat{\Gamma}$, we have
$$ e(\Gamma,G)  =  \frac{2^2 (p_1-1) p_2 p_3}{p_1-1} \left( 2 \cdot
 \frac{p_2+1}{2p_2} \cdot\frac{p_3+1}{2p_3} + (p_1-3) \right) $$
while for $G \not \cong \Gamma$ or $\widehat{\Gamma}$, we get 
$$  e(\Gamma,G)  =  \frac{2^2 (p_1-1) p_2 p_3}{p_1-1} \left( 2 \cdot
 \frac{p_2+1}{2p_2} +2 \cdot\frac{p_3+1}{2p_3} + (p_1-5) \right). $$
Hence we have (\ref{lastcase}).
\end{proof}

\begin{remark}
Our results agree with those of \cite{kohl13,
  kohl16}. When $p_1=2$ and $p_3=2p_2+1$, the groups corresponding to
the factorisations in Table \ref{3-prime-table} are denoted in
\cite{kohl13} as $C_{mp}$, $C_p \times D_q$, $C_q \times D_p$,
$D_{pq}$, $F \times C_2$, $\Hol\ C_p$, respectively. Our results in
Table \ref{HGS-2} in this case then agree with \cite[Theorem
  5.1]{kohl13}. When $p_1>2$, with $p_3 \equiv p_2 \equiv 1
\pmod{p_1}$ but $p_3 \not \equiv 1 \pmod{p_2}$, only 
Factorisations $1$--$4$ in our Table \ref{3-prime-table} occur. In
\cite{kohl16}, the corresponding groups are denoted respectively by
$C_{p_3 p_2 p_1}$, $C_{p_3} \times (C_{p_2} \rtimes C_{p_1})$,
$C_{p_2} \times (C_{p_3} \rtimes C_{p_1})$, $C_{p_3 p_2} \rtimes
C_{p_1}$, where the last case gives $p_1-1$ isomorphism classes of
groups, coming from different actions of $C_{p_1}$ on $C_{p_3
  p_2}$. Our Table \ref{HGS-not2} (with the last two rows and columns
omitted) matches the first table in \cite[Theorem 2.4]{kohl16}, and
the two cases in (\ref{lastcase}) respectively match the cases $j=i$,
$-i$ and $j \neq i $, $-i$ in the second table there.
\end{remark}

\bibliography{GSQ-bib}

\providecommand{\bysame}{\leavevmode\hbox to3em{\hrulefill}\thinspace}
\providecommand{\MR}{\relax\ifhmode\unskip\space\fi MR }
\providecommand{\MRhref}[2]{%
  \href{http://www.ams.org/mathscinet-getitem?mr=#1}{#2}
}
\providecommand{\href}[2]{#2}
\begin{thebibliography}{CRV18}

\bibitem[AB]{AB-braces}
Ali~A. Alabdali and Nigel~P. Byott, \emph{Skew braces of squarefree order},
  Preprint, \verb+arXiv:1910.07814+.

\bibitem[AB18]{AB}
\bysame, \emph{Counting {H}opf-{G}alois structures on cyclic field extensions
  of squarefree degree}, J. Algebra \textbf{493} (2018), 1--19. \MR{3715201}

\bibitem[Ala18]{thesis}
Ali~A. Alabdali, \emph{Hopf-{G}alois structures on {G}alois extensions of
  fields of squarefree degree}, Ph.D. thesis, University of Exeter, 2018.

\bibitem[Bac16]{Bachiller}
David Bachiller, \emph{Counterexample to a conjecture about braces}, J. Algebra
  \textbf{453} (2016), 160--176. \MR{3465351}

\bibitem[Byo96]{unique}
N.~P. Byott, \emph{Uniqueness of {H}opf {G}alois structure for separable field
  extensions}, Comm. Algebra \textbf{24} (1996), no.~10, 3217--3228.
  \MR{1402555}

\bibitem[Byo04]{pq}
Nigel~P. Byott, \emph{Hopf-{G}alois structures on {G}alois field extensions of
  degree {$pq$}}, J. Pure Appl. Algebra \textbf{188} (2004), no.~1-3, 45--57.
  \MR{2030805}

\bibitem[CC99]{CC}
Scott Carnahan and Lindsay Childs, \emph{Counting {H}opf {G}alois structures on
  non-abelian {G}alois field extensions}, J. Algebra \textbf{218} (1999),
  no.~1, 81--92. \MR{1704676}

\bibitem[Chi03]{Ch03}
Lindsay~N. Childs, \emph{On {H}opf {G}alois structures and complete groups},
  New York J. Math. \textbf{9} (2003), 99--115. \MR{2016184}

\bibitem[CRV18]{CRV-symalt}
Teresa Crespo, Anna Rio, and Montserrat Vela, \emph{Hopf {G}alois structures on
  symmetric and alternating extensions}, New York J. Math. \textbf{24} (2018),
  451--457. \MR{3855635}

\bibitem[CS69]{CS}
Stephen~U. Chase and Moss~E. Sweedler, \emph{Hopf algebras and {G}alois
  theory}, Lecture Notes in Mathematics, Vol. 97, Springer-Verlag, Berlin-New
  York, 1969. \MR{0260724}

\bibitem[GP87]{GP}
Cornelius Greither and Bodo Pareigis, \emph{Hopf {G}alois theory for separable
  field extensions}, J. Algebra \textbf{106} (1987), no.~1, 239--258.
  \MR{878476}

\bibitem[GV17]{GV}
L.~Guarnieri and L.~Vendramin, \emph{Skew braces and the {Y}ang-{B}axter
  equation}, Math. Comp. \textbf{86} (2017), no.~307, 2519--2534. \MR{3647970}

\bibitem[H{\"o}l95]{holder}
Otto H{\"o}lder, \emph{Die {G}ruppen mit quatratfreier {O}rdnungszahl},
  Nachr.~K\"onigl.~Ges.~Wiss.~G\"ottingen Math. Phys. \textbf{1} (1895),
  211--229.

\bibitem[Koh98]{kohl}
Timothy Kohl, \emph{Classification of the {H}opf {G}alois structures on prime
  power radical extensions}, J. Algebra \textbf{207} (1998), no.~2, 525--546.
  \MR{1644203}

\bibitem[Koh13]{kohl13}
\bysame, \emph{Regular permutation groups of order {$mp$} and {H}opf {G}alois
  structures}, Algebra Number Theory \textbf{7} (2013), no.~9, 2203--2240.
  \MR{3152012}

\bibitem[Koh16]{kohl16}
\bysame, \emph{Hopf-{G}alois structures arising from groups with unique
  subgroup of order {$p$}}, Algebra Number Theory \textbf{10} (2016), no.~1,
  37--59. \MR{3463035}

\bibitem[MM84]{MM}
M.~Ram Murty and V.~Kumar Murty, \emph{On groups of squarefree order}, Math.
  Ann. \textbf{267} (1984), no.~3, 299--309. \MR{738255}

\bibitem[NZ19]{NZ}
Kayvan Nejabati~Zenouz, \emph{Skew braces and {H}opf-{G}alois structures of
  {H}eisenberg type}, J. Algebra \textbf{524} (2019), 187--225. \MR{3905210}

\bibitem[Rum07]{Rump}
Wolfgang Rump, \emph{Braces, radical rings, and the quantum {Y}ang-{B}axter
  equation}, J. Algebra \textbf{307} (2007), no.~1, 153--170. \MR{2278047}

\bibitem[SV18]{SV}
Agata Smoktunowicz and Leandro Vendramin, \emph{On skew braces (with an
  appendix by {N}. {B}yott and {L}. {V}endramin)}, J. Comb. Algebra \textbf{2}
  (2018), no.~1, 47--86. \MR{3763907}

\end{thebibliography}
\end{document}